\newif\iffull\fulltrue

\iffull
\documentclass[11pt]{article}
\else
\documentclass[cleveref,english]{socg-lipics-v2021}
\fi

\usepackage[T1]{fontenc}
\usepackage[utf8]{inputenc}

\iffull
\usepackage[margin=1in]{geometry}
\fi

\usepackage{amsfonts,amsmath,amssymb,amsthm}
\usepackage{cite,microtype,graphicx,hyperref}

\iffull
\usepackage[capitalize,nameinlink]{cleveref}

\newtheorem{theorem}{Theorem}
\newtheorem{lemma}{Lemma}
\newtheorem{corollary}[lemma]{Corollary}
\newtheorem{observation}[lemma]{Observation}

\theoremstyle{definition}
\newtheorem{definition}{Definition}
\newtheorem{example}{Example}

\fi

\title{Non-Euclidean Erd\H{o}s--Anning Theorems}

\iffull
\author{David Eppstein\thanks{Department of Computer Science, University of California, Irvine. Donald Bren Hall, Irvine, CA 92697, USA. eppstein@uci.edu.}}
\else
\author{David Eppstein}{University of California, Irvine \and\url{https://www.ics.uci.edu/~eppstein/}}{eppstein@uci.edu}{}{Research supported in part by NSF grant CCF-2212129.}
\authorrunning{D. Eppstein}
\Copyright{David Eppstein}



\relatedversiondetails{Full
Version}{https://arxiv.org/2401.06328}

\EventEditors{Oswin Aichholzer and Haitao Wang}
\EventNoEds{2}
\EventLongTitle{41st International Symposium on Computational Geometry
(SoCG 2025)}
\EventShortTitle{SoCG 2025}
\EventAcronym{SoCG}
\EventYear{2025}
\EventDate{June 23--27, 2025}
\EventLocation{Kanazawa, Japan}
\EventLogo{socg-logo.pdf}
\SeriesVolume{332}
\ArticleNo{XX}
\fi

\begin{document}
\maketitle
\begin{abstract}
The Erd\H{o}s--Anning theorem states that every point set in the Euclidean plane with integer distances must be either collinear or finite. More strongly, for any (non-degenerate) triangle of diameter~$\delta$, at most $O(\delta^2)$ points can have integer distances from all three triangle vertices. We prove the same results for any strictly convex distance function on the plane, and analogous results for every two-dimensional complete Riemannian manifold of bounded genus and for geodesic distance on the boundary of every three-dimensional Euclidean convex set. As a consequence, we resolve a 1983 question of Richard Guy on the equilateral dimension of Riemannian manifolds. Our proofs are based on the properties of additively weighted Voronoi diagrams of these distances.
\end{abstract}

\section{Introduction}

Many incidence properties of lines or curves, seemingly mysterious on their own, become obvious when viewed through the lens of computational geometry, where the curves may be features of Voronoi diagrams or related constructions. Thus, for instance, to see that the three perpendicular bisectors of a triangle's sides meet at a single point, observe that these bisectors contain the edges of the Voronoi diagram of the triangle vertices, and these Voronoi edges meet at the Voronoi vertex. In this work we apply the same principle to the Erd\H{o}s--Anning theorem on integer distances in the plane, conventionally proven using the algebraic intersection properties of hyperbolae. By reinterpreting these hyperbolae and their intersection points as the cell boundaries and vertices of additively-weighted Voronoi diagrams, and by analyzing these diagrams geometrically and topologically rather than algebraically, we greatly generalize this theorem, to several wide classes of two-dimensional metric spaces. In doing so, we resolve a 1983 question of Richard Guy on the equilateral dimension of Riemannian manifolds, the maximum number of points that can be placed at equal distances from each other in these spaces.

The Erd\H{o}s--Anning theorem states that every point set in the Euclidean plane with integer distances must be either collinear or finite~\cite{AnnErd-BAMS-45,Erd-BAMS-45}. More strongly, for any (non-degenerate) triangle of diameter $\delta$, at most $O(\delta^2)$ points can have integer distances from all three triangle vertices. When the whole point set has diameter $D$, its number of points is $O(D)$~\cite{Sol-DCG-03}, and if in addition it is non-collinear its number of points is $D^{O(1/\log\log D)}$~\cite{GreIliPel-24}. In contrast, as Euler already observed, there exist infinite non-collinear point sets of bounded diameter with all distances rational, for instance, dense subsets of a unit circle~\cite{Eul-FxA-62,Har-CH-98}. 

There are many metrics other than Euclidean distance for which we might ask similar questions. One obstacle to generalizing the Erd\H{o}s--Anning theorem to other metrics is that its usual proof is heavily based on algebraic geometry. In the Euclidean plane, the locus of points whose distances to two given points differ by a fixed integer is a branch of a hyperbola. Therefore, the points with integer distances to three non-collinear points lie at the intersection points of two finite families of hyperbolas. Hyperbolas are algebraic curves of degree two, and by Bezout's theorem two hyperbolas have at most four crossings, so these two finite families have finitely many intersection points. But it is not obvious how to extend this argument to metrics for which the corresponding curves are more complicated or non-algebraic. Erd\H{o}s and Anning originally used a different argument using trigonometric inequalities~\cite{AnnErd-BAMS-45}, but its generalization is also non-obvious.

A second obstacle is that analogous statements are untrue for some metrics. For the $L_1$ or $L_\infty$ distances in the plane, the integer lattice provides a familiar example of an infinite set that is not collinear and has only integer distances. Thus, generalizations of the theorem to other families of metrics must avoid families that include $L_1$ or $L_\infty$.

In this paper we prove the following generalized versions of the Erd\H{o}s--Anning theorem:

\begin{quotation}
\noindent\textit{For any strictly convex distance function on $\mathbb{R}^2$, every point set with integer distances must be either collinear or finite. For any (non-degenerate) triangle of diameter $\delta$, at most $O(\delta^2)$ points can have integer distances from all three triangle vertices. If a set with integer distances has diameter $D$, it has at most $O(D)$ points.} (\cref{sec:dist-fun})

\smallskip
\noindent\textit{For any two-dimensional complete Riemannian manifold of finite genus $g$, every point set with integer distances must either be finite or be contained in a geodesic (a locally-shortest curve) that contains every shortest curve between pairs of the points. If some three points do not lie on a single shortest curve, and have diameter $\delta$, at most $O((g+1)\delta^2)$ points can have integer distances from these three points.} (\cref{sec:riemann})

\smallskip
\noindent\textit{For geodesic distance on the boundary of a convex set in $\mathbb{R}^3$, every point set with integer distances must either be finite or be contained in a geodesic that contains every shortest curve between pairs of the points. If some three points do not lie on a single shortest curve, and have diameter $\delta$, at most $O(\delta^2)$ points can have integer distances from these three points. If a set with integer distances has diameter $D$, it has at most $O(D^{4/3})$ points.} (\cref{sec:convex})
\end{quotation}

Our proofs replace the intersection properties of low-degree algebraic curves with the geometric and topological properties of additively weighted Voronoi diagrams. Our results for Riemannian manifolds resolve a question posed in 1983 by Richard Guy~\cite{Guy-AMM-83}, on the \emph{equilateral dimension} of Riemannian manifolds, the maximum number of points at unit distance from each other in these manifolds. Guy asked whether the equilateral dimension could be bounded by a function of dimension. For complete Riemannian 2-manifolds of bounded genus, we bound  the equilateral dimension by a function of the genus (\cref{cor:equilateral-dimension}) but for locally Euclidean incomplete Riemannian 2-manifolds,  for complete Riemannian 2-manifolds of unbounded genus, and for complete Riemannian metrics on~$\mathbb{R}^3$, we show that the equilateral dimension is unbounded (\cref{ex:infinite-cone} and \cref{ex:high-genus}).

\section{Preliminaries: Voronoi diagrams and star-shaped cells}

A \emph{Voronoi diagram} of a finite set of point \emph{sites} in a metric space $(X,d)$ subdivides the space according to which site is closest. Each site $s_i$ has a (closed) cell $V_i$ associated with it, where
\[
V_i=\left\{ p\in X \biggm\vert d(p,s_i)\le \min_{j\ne i} d(p,s_j)\right\}
\]
is the set of points whose distance to site $s_i$ is less than or equal to the distance to any other site $s_j$. We will only use Voronoi diagrams of finite sets (in fact, of at most three sites), so we only define and discuss the properties of these diagrams for finite sets of sites.

The intersection of any collection of Voronoi cells is the locus of points having equal distances to their sites. When $X$ has the topology of the plane, the \emph{bisector} of a pair of sites (the intersection of their Voronoi cells) is often (but not always) a curve, and a non-empty intersection of three or more sites is often (but not always) an isolated point or points, called \emph{Voronoi vertices}. We will use \emph{additively weighted} Voronoi diagrams, in which each site $s_i$ has a weight $w_i$ associated with it, and in which the distances in the definition of cells are modified by adding these weights:
\[
V_i=\left\{ p\in X \biggm\vert d(p,s_i)+w_i\le \min_{j\ne i} d(p,s_j)+w_j\right\}.
\]
Voronoi diagrams have been the object of intensive study, including Voronoi diagrams for convex distance functions~\cite{CheDry-SoCG-85,IckKleLe-FI-95,Le-IPL-97}, for points on a sphere~\cite{Mil-San-71,AugPes-JCP-85,NaLeeChe-CGTA-02}, for points in the hyperbolic plane~\cite{DevMeiTei-CCCG-92,NieNoc-ICCSA-10,BogDevTei-JoCG-14}, orbifolds~\cite{MazRec-CGTA-97}, and for more general Riemannian manifolds~\cite{LeiLet-SoCG-00}. Additively weighted Voronoi diagrams are also a standard concept~\cite{CheDry-SICOMP-81,MenEmi-SODA-03}, and confusingly, are sometimes called \emph{hyperbolic Dirichlet tessellations} despite their use of Euclidean distance~\cite{AshBol-GD-86}, because their cells are bounded by hyperbolas (the same hyperbolas from the Erd\H{o}s--Anning proof). We did not find works combining additive weights and non-Euclidean distances.

\begin{figure}[t]
\centering\includegraphics[width=0.6\textwidth]{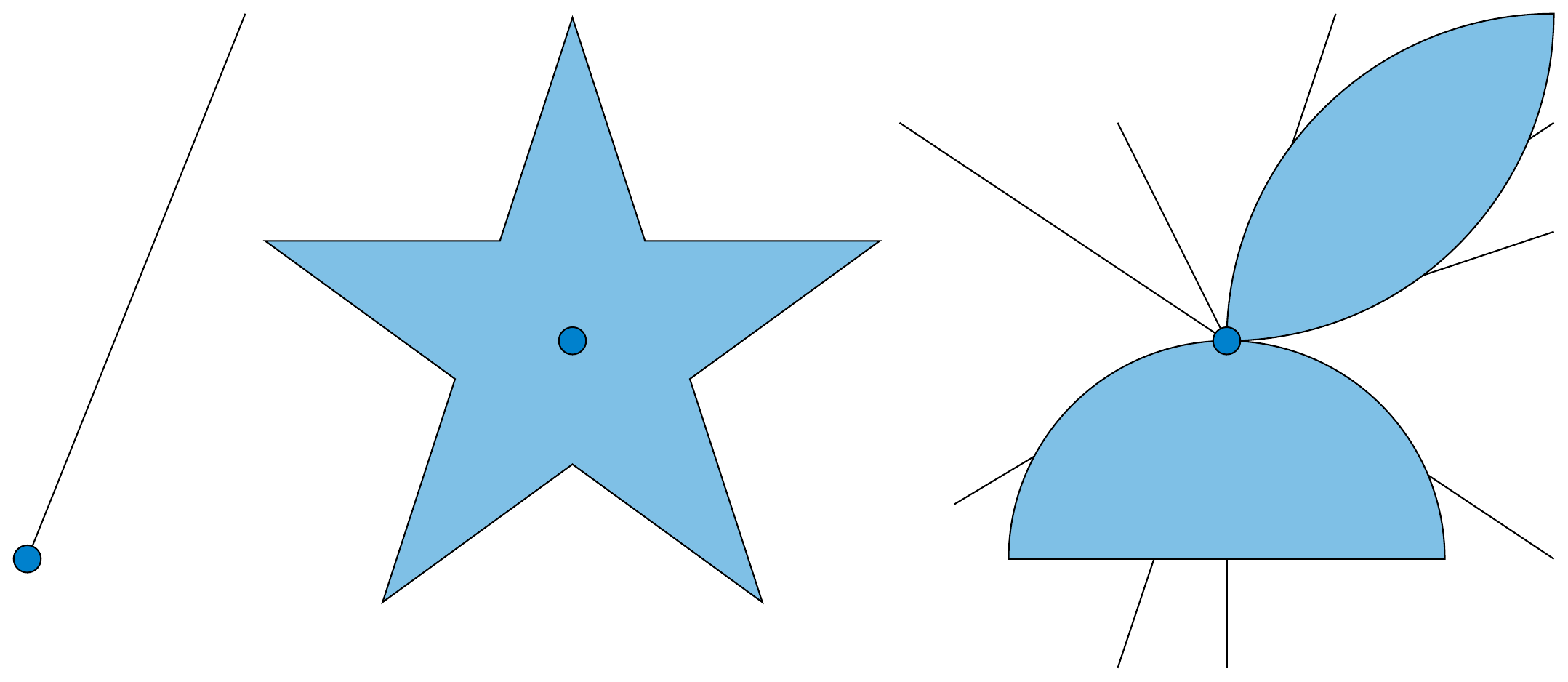}
\caption{Three star-shaped sets, each shown with a point in its kernel.}
\label{fig:star-shaped}
\end{figure}

The \emph{kernel} of a subset $S$ of the plane is defined as the set of points $p$ such that, for every point~$q$ in $S$, the line segment $pq$ belongs to $S$. $S$ is convex if and only if its kernel coincides with $S$. Weakening this formulation of convexity, $S$ is defined to be \emph{star-shaped} if its kernel is non-empty. (See \cref{fig:star-shaped}. Note that, according to this definition, star-shaped sets are not required to form topological disks, and the empty set is not star-shaped.) It is known that, for instance, the Voronoi cells of unweighted convex distance functions are star-shaped topological disks~\cite{CheDry-SoCG-85}. For additively weighted Voronoi diagrams, we should be a little more careful, because of the possibility that a Voronoi cell can degenerate to a line segment, a single point, or the empty set, but in the Euclidean case they are again known to be star-shaped when non-empty~\cite{AshBol-GD-86}.

\section{Convex distance functions}
\label{sec:dist-fun}

A \emph{convex distance function} on $\mathbb{R}^d$ can be defined from any centrally symmetric compact convex body~$K$. The resulting function $d_K(p,q)$ is the smallest scale factor $s$ such that a copy of $K$, scaled by~$s$ and centered at~$p$, contains~$q$. Any convex distance function defines a norm on $\mathbb{R}^d$, $||p||_K=d_K(0,p)$. In the reverse direction, the closed unit ball $\{p\mid||p||\le 1\}$ of any norm $||\cdot||$ can be used to define a convex distance function. A convex set is \emph{strictly convex} if its boundary does not contain any line segment, and a \emph{strictly convex distance function} is a convex distance function defined from a strictly convex set. Convex distance functions obey all the requirements of a metric space, including the triangle inequality. For strictly convex distance functions and non-collinear triples of points $p$, $q$, $r$, they obey a strict form of the triangle inequality~\cite{CheDry-SoCG-85}:
\[
d_K(p,r)<d_K(p,q)+d_K(q,r).
\]
Examples of strictly convex distance functions include the Minkowski $L_p$ distances for $1<p<\infty$, obtained from the convex body
\[
K_p = \left\{ (x_1,x_2,\dots x_d) \biggm\vert \sum_{i=1}^d x_i^p \le 1 \right\}.
\]
For $p=1$ and for the limiting case $p=\infty$ one obtains distance functions that are convex but not strictly convex. For $p=2$ we recover the familiar Euclidean distance.
\iffull

A \emph{curve}, in one of these spaces, is just a continuous function from the unit interval $[0,1]$ to the space. Its \emph{endpoints} are the values of this function at $0$ and $1$. If $C$ is a curve for a space with distance $d$, we may define its length as
\[
\lim_{n\to \infty} \sum_{i=1}^n d\left(C\left(\frac{i-1}{n}\right),C\left(\frac{i}{n}\right)\right),
\]
when this limit exists.
\fi
Convex distance functions define \emph{geodesic metric spaces} in which the distance between any two points equals the length of a shortest curve between the points. For strictly convex distance functions, this shortest curve is unique: it is just the \emph{line segment} between its two points.
\iffull
It is convenient to combine these observations about shortest curves with the triangle inequality to obtain the following observation about distances:

\begin{observation}[Lipschitz property of distances]
\label{obs:lipschitz}
Let $p$, $q$, and $r$ be points of a geodesic space with distance $d$, and let $C$ be a curve with endpoints $q$ and $r$, of length $\ell$. Then
\[
-\ell\le d(p,q)-d(p,r)\le\ell,
\]
with equality on the left side only when $C$ traces without backtracking a subset of a shortest curve from $p$ to $r$, and with
equality on the right side only when $C$ traces without backtracking a subset of a shortest curve from $q$ to $p$.
\end{observation}

Thus, when the space is given by a strictly convex distance function, and when the difference of distances equals the length of the curve $C$, the three points $p$, $q$, and $r$ must be collinear, and curve~$C$ must form a segment of their line.
\fi

\iffull
\subsection{Examples}
In this section we provide examples of integer distance sets for different convex distance functions.

\begin{figure}[t]
\centering\includegraphics[scale=0.42]{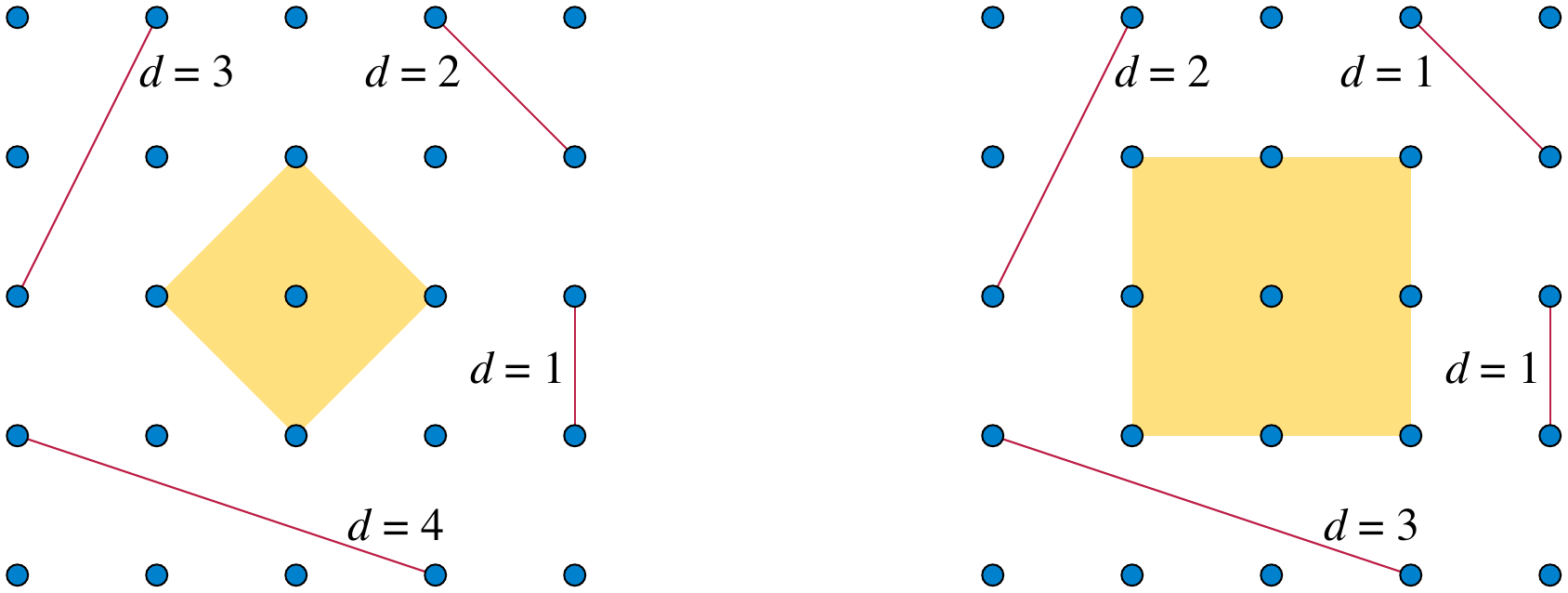}
\caption{Integer $L_1$ (left) and $L_\infty$ distances in the integer grid (\cref{ex:grid-distances}). The yellow shaded regions are the unit balls for these two distances.}
\label{fig:grid-distances}
\end{figure}

\begin{example}
\label{ex:grid-distances}
The integer grid $\{(x,y)\mid x,y\in\mathbb{Z}\}$ has integer distances for the $L_1$ distance (the sum of absolute differences of coordinates) and $L_\infty$ distance (maximum absolute difference of coordinates); see \cref{fig:grid-distances}. Both the $L_1$ and $L_\infty$ distance are convex distance functions, but are not strictly convex distance functions. The unit ball of the $L_1$ distance is a $45^\circ$-rotated square with vertices $(0,\pm 1)$ and $(\pm 1,0)$, while the unit ball of $L_\infty$ distance is an axis-parallel square with vertices $(\pm 1,\pm 1)$.
\end{example}

\begin{example}
\label{ex:complete}
Let $S$ be any finite set of points in the plane for which all $\tbinom{n}{2}$ lines determined by $S$ have distinct slopes. Then there exists a strictly convex distance function $d_K$ that assigns integer distances to each pair of points in $S$. To construct $d_K$, consider the system of unit vectors obtained by normalizing the vector differences of each pair of points; these form a centrally symmetric set on the unit circle, and by assumption they are distinct. Let $\varepsilon$ be a sufficiently small number that expanding any subset of these unit vectors by a factor of $1+\varepsilon$ would preserve the strict convex position of all the vectors. Scale the unit circle and the vectors on it so that, when using the result as the unit ball for a convex distance function, all distances in $S$ become larger than $1/\varepsilon$. Round each of these distances down to an integer, and expand each of the corresponding vectors by the same factor (at most $1+\varepsilon$) by which its distance was decreased. Construct $K$ as any strictly convex set through the resulting vectors, for instance by connecting them in cyclic order using circular arcs of sufficiently large radius. Then for this choice of $K$, each distance in $S$ becomes its rounded value, which by construction is an integer.
\end{example}

A similar construction shows that any countable set of points for which all lines have distinct slopes can be made to have rational (but not integer) distances for a strictly convex distance function, obtained by adding the points one at a time and perturbing the distance function for each added point without changing the previous distances.

\begin{figure}[t]
\centering\includegraphics[scale=0.333]{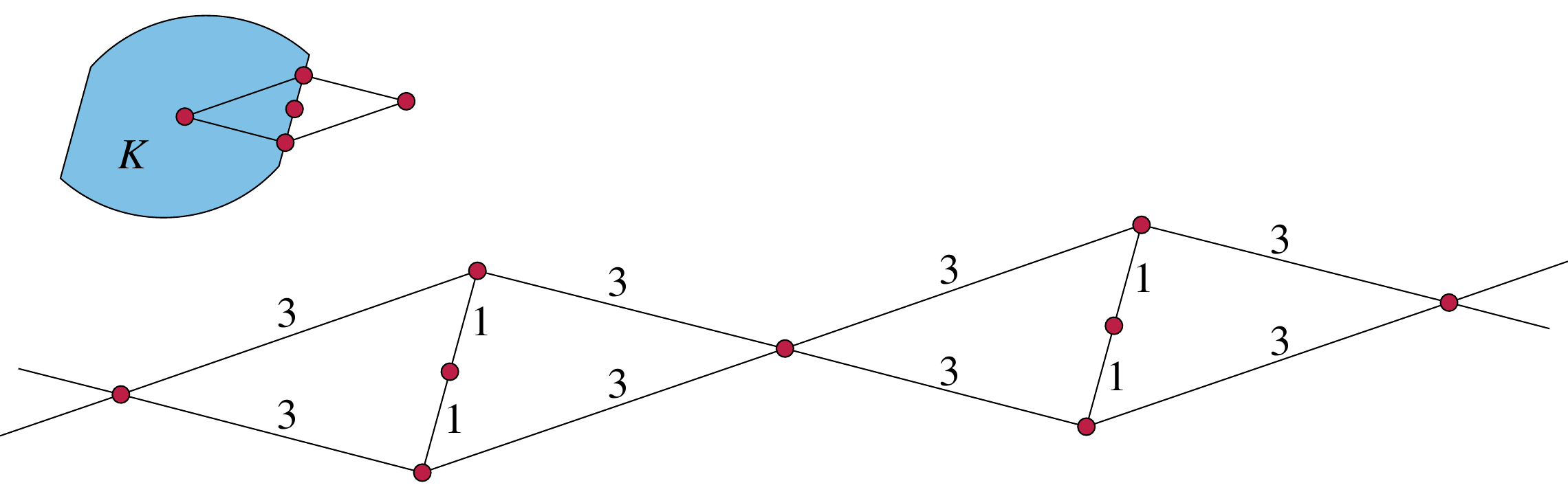}
\caption{Construction for an infinite non-geodesic set with integer distances for a convex but not strictly convex distance function $d_K$ (\cref{ex:non-geodesic}).}
\label{fig:non-geodesic}
\end{figure}

\begin{example}
\label{ex:non-geodesic}
For any convex distance functions $d_K$ that is not strict, there exist infinite sets that have integer distances but that do not belong to a single shortest curve. To construct such a set, note that if $d_K$ is convex but not strictly convex, then the boundary of $K$ contains a line segment. Place two or more points along this segment, at equal rational distances~$\delta$ according to $d_K$, and add two more points at the origin and at its reflection across this line segment, each having unit distance to the entire line segment according to $d_K$. Scale this system of points by the denominator of $\delta$ to produce a finite set of points with integer distances, having a parallelogram as its convex hull. Connecting multiple translated copies of this parallelogram by identifying their copies of the origin and its reflection (\cref{fig:non-geodesic}) produces the desired infinite non-geodesic set with integer distances.
\end{example}

\begin{example}
Euclidean distance is a convex distance function, and there exist infinite dense sets of points on the unit circle with rational Euclidean distances~\cite{Eul-FxA-62,Har-CH-98}. One construction for such a system of points uses complex number coordinates for the points of the plane, and consists of the complex numbers $\{q^{2j}\mid j\in\mathbb{Z}\}$ where $q=\tfrac{a}{c}+i\tfrac{b}{c}$ is a rational point on the unit circle derived from an integer right triangle with side lengths $a$, $b$, and $c$. Geometrically, this point set is obtained from a single point by repeatedly rotating by twice the angle of the triangle.\footnote{This construction is essentially due to Harborth~\cite{Har-CH-98} but he states it incorrectly as using all integer multiples of the angle rather than only the even multiples.} \cref{fig:345-rotations} shows an example derived from the 3-4-5 right triangle.
Arbitrarily large finite non-collinear point sets at integer distances can be obtained by taking a finite subset of this infinite set (possibly including also the center of the circle) and scaling
it to clear the denominators in its coordinates.
\end{example}

\begin{figure}[t]
\centering\includegraphics[width=0.35\textwidth]{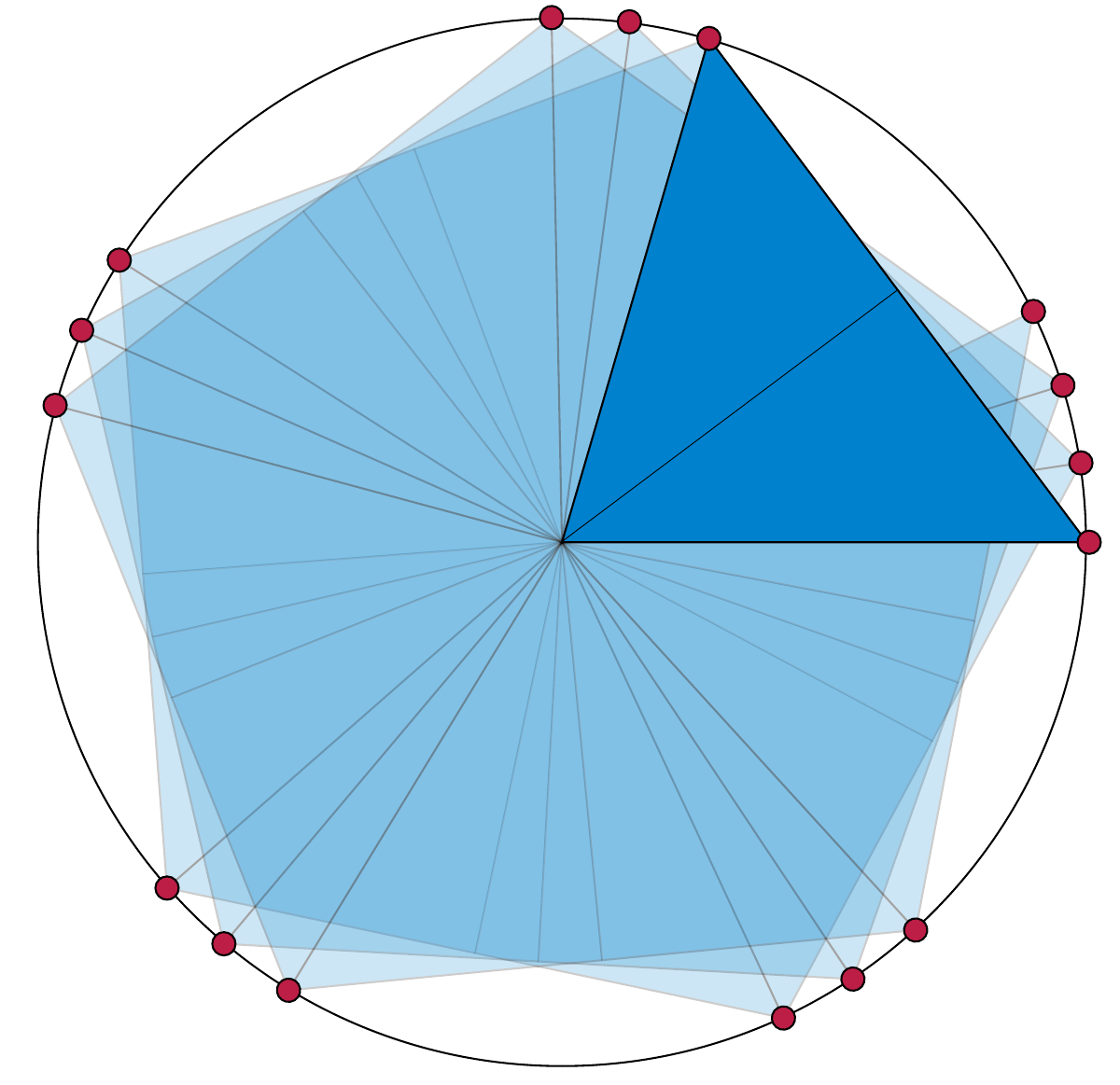}
\caption{The points on the unit sphere whose angles are even multiples of the angles of a 3-4-5 right triangle have pairwise rational distances.}
\label{fig:345-rotations}
\end{figure}

It is an open problem whether there exist arbitrarily-large sets of points at integer Euclidean distances that have no three points in a line and no four points on a circle~\cite{Epp-18}. The largest such sets known have only seven points~\cite{Kle-MI-08,KreKur-DCG-08}.
\fi

\iffull
\subsection{Degenerate and non-degenerate sites}
\label{sec:degenerate}

In unweighted Voronoi diagrams, every cell is a topological disk (or the whole space), having its site in its interior. That is not true for  additively weighted Voronoi diagrams, but the cells that do not have their site in their interior have a very special structure, which we examine in this section.

\begin{lemma}
\label{lem:star}
For additively weighted Voronoi diagrams of 2-dimensional strictly convex distance functions, every non-empty Voronoi cell is star-shaped with its site in its kernel.
\end{lemma}

\begin{proof}
The star-shaped property is an immediate consequence of the Lipschitz property of distances. If site $s_i$ has the minimum weighted distance to a given point $p$, then for any other point $q$ on line segment $ps_i$, the weighted distance from $s_i$ to $q$ is less by $d(p,q)$, while other weighted distances decrease by at most the same amount, so the weighted distance from $s_i$ remains minimum. Thus, if~$p$ belongs to the Voronoi cell of $s_i$, so does every point $q$ on line segment $ps_i$.
\end{proof}

\begin{definition}
Define a \emph{degenerate site} of an additively weighted Voronoi diagram of a 2-dimensional strictly convex distance function to be a site that belongs to the Voronoi cell of another site. Define a \emph{non-degenerate site} to be a site that is not degenerate.
\end{definition}

\begin{figure}[t]
\centering\includegraphics[scale=0.42]{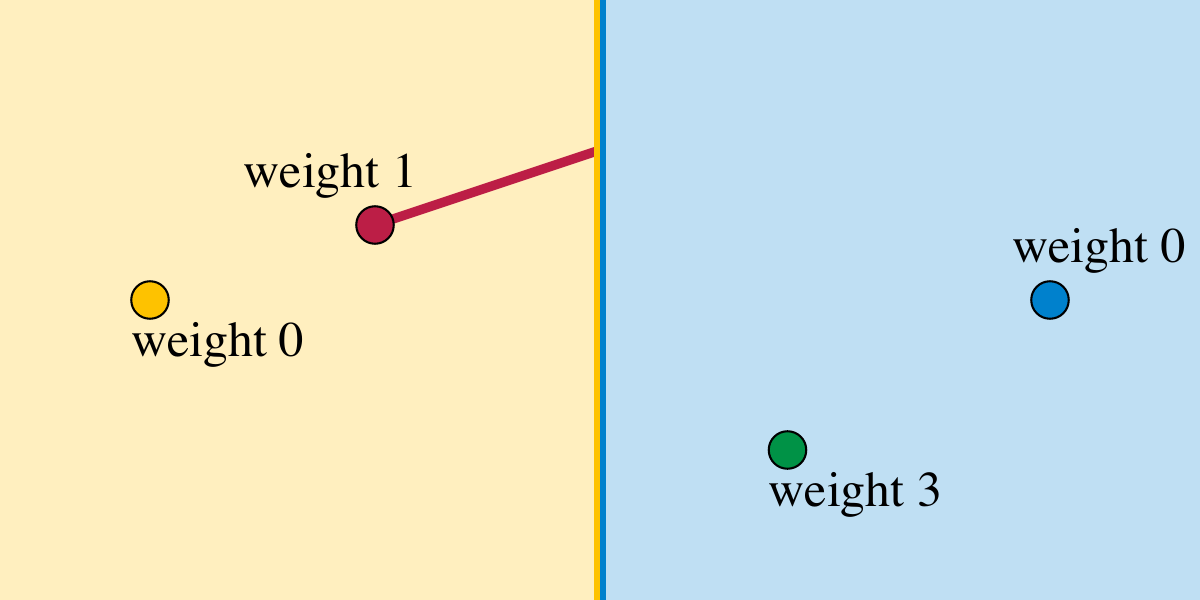}
\caption{An additively weighted Voronoi diagram with two degenerate sites (red and green) and two non-degenerate sites (yellow and blue). The red site has equal weighted distance to itself and to the yellow site; its Voronoi cell (shown as a red line segment) extends along a ray, away from the yellow site, to the boundary of the yellow Voronoi cell. The green site has larger weighted distance to itself than to the blue site, so its Voronoi cell is empty. The Voronoi cells for the yellow and blue sites together cover the plane, separated from each other by the central vertical line. The distance function is unspecified; it could be Euclidean, or any other strictly convex distance function whose unit disk has vertical reflection symmetry (so that the bisector of the two horizontally-aligned yellow and blue points is a vertical line).}
\label{fig:degenerate}
\end{figure}

For examples of degenerate and non-degenerate sites, see \cref{fig:degenerate}. The next lemma explains the shapes of the cells in this example:

\begin{lemma}
\label{lem:ray}
Let $s_i$ be a degenerate site of an additively weighted Voronoi diagram of a 2-dimensional strictly convex distance function, with $s_i$ belonging to the Voronoi cell $V_j$ of another site $s_i$. Let $R$ be a ray on line $s_is_j$ with $s_i$ as its initial point, extending away from~$s_j$.
 Then either $s_i$ is empty, or $V_i=R\cap V_j$.
\end{lemma}

\begin{proof}
If the weighted distance from $s_i$ to $s_j$ is less than the weighted distance from $s_i$ to itself, then $s_i$ does not belong to its own Voronoi cell, and the only possibility according to \cref{lem:star} is that $V_i$ is empty. Otherwise, at $s_i$, the weighted distances to $s_i$ and $s_j$ are equal. Let $R$ be as described in the statement of the lemma. By the Lipschitz property of distances, every point of $R$ has equal weighted distance to $s_i$ and to $s_j$, and every other point in the plane has smaller weighted distance to $s_j$ than to $s_i$. Therefore, the set of points whose weighted distance to $s_i$ is minimum (among weighted distances to all sites) is exactly $R\cap V_j$.
\end{proof}

By repeatedly using \cref{lem:ray} to replace the Voronoi cells of degenerate sites with supersets that are also Voronoi cells, until no more such replacements can be performed, we obtain:

\begin{corollary}
\label{cor:non-degenerate-cover}
In an additively weighted Voronoi diagram of a 2-dimensional strictly convex distance function, every point belongs to the Voronoi cell of at least one non-degenerate site.
\end{corollary}

Note that \cref{cor:non-degenerate-cover} would not be true if we allowed infinite sets of sites. Instead of defining degeneracy of a site by its membership in other cells, we have an equivalent characterization by its position within its own cell:

\begin{lemma}
\label{lem:non-degenerate-interior}
A site $s_i$ of an additively weighted Voronoi diagram of a 2-dimensional strictly convex distance function is non-degenerate if and only if it is an interior point of its cell $V_i$.
\end{lemma}

\begin{proof}
If $s_i$ is an interior point of $V_i$, $V_i$ cannot be a subset of a ray (which has no interior points), and it follows from \cref{lem:ray} that $s_i$ is non-degenerate. If $s_i$ is not an interior point of $V_i$, then every neighborhood of $s_i$ has points that are not in $V_i$, and therefore there exists some other cell~$V_j$ intersecting this neighborhood. Because there are only finitely many sites, we can reverse the quantifiers: there exists some other cell $V_j$ that intersects every neighborhood of $s_i$. But because cells are closed sets, this implies that $s_i\in V_j$ and $s_i$ is degenerate.
\end{proof}

As the next lemma shows, each ray with a non-degenerate site as its initial point passes through at most one boundary point of its Voronoi cell.

\begin{figure}[t]
\centering\includegraphics[scale=0.42]{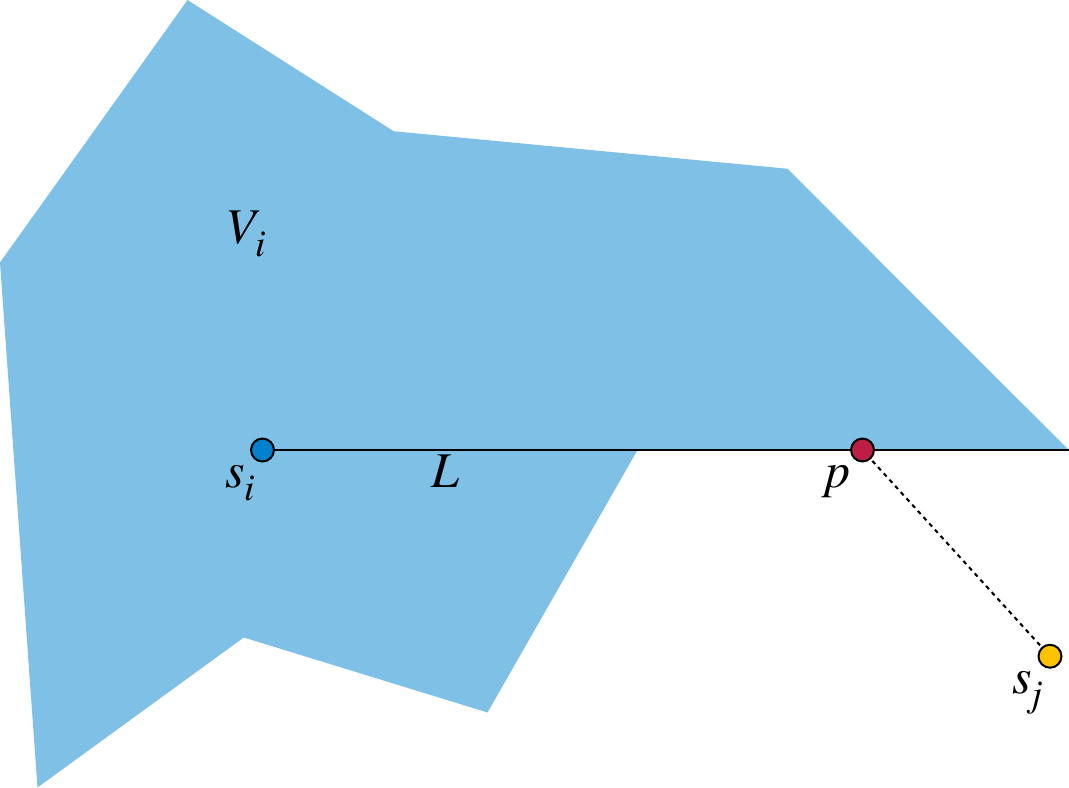}
\caption{Figure for the proof of \cref{lem:interior}: a site $s_i$, a segment $L$ within Voronoi cell $V_i$, and a point $p$ in the relative interior of $L$ that is not interior to $V_i$. The lemma proves that this situation cannot exist by considering cases for the location of $s_j$, another site with the same distance to $p$ as $s_i$.}
\label{fig:interior}
\end{figure}

\begin{lemma}
\label{lem:interior}
Let $s_i$ be a non-degenerate site of an additively weighted Voronoi diagram of a 2-dimensional strictly convex distance function, let $L$ be a line segment in $V_i$ having $s_i$ as an endpoint, and let $p$ be a point in the relative interior of $L$. Then $p$ is interior to $V_i$.
\end{lemma}

\begin{proof}
Suppose for a contradiction that $p$ is not interior to $V_i$. Then by finiteness of the number of sites, there must exist another site $s_j$ such that every neighborhood of $p$ contains points of $V_j\setminus V_i$. Because Voronoi cells are closed, $p$ belongs to $V_j$ as well as to $V_i$, and has equal weighted distance to $s_i$ and $s_j$.
We now distinguish three cases for the possible location of $s_j$:
\begin{itemize}
\item Suppose that $s_j$ does not lie on the line containing $L$, as shown in \cref{fig:interior}, or lies on the opposite side from $p$ as $s_i$ on this line or great circle. Then by the Lipschitz property of distances, all points of $L$ that are on the opposite side from $p$ as $S_i$ would have smaller weighted distance to $s_j$ than to $s_i$, because the distance from $s_i$ increases at the same rate as distance from $p$ along $L$, while the distance from $s_j$ cannot increase as quickly. But this contradicts the assumption that some of these points belong to $V_i$.
\item Suppose that $s_j$ lies on segment or arc $ps_i$. Then $s_j\in V_i$, and by \cref{lem:ray} its Voronoi cell $V_j$ is a subset of $V_i$, contradicting the assumption that $N$ contains a point of $V_j\setminus V_i$.
\item The only remaining case is that $s_i$ lies on line segment $ps_j$. But this lies within $V_j$ by \cref{lem:star}, so $s_i\in V_j$, contradicting the assumption that $s_i$ is non-degenerate.
\end{itemize}
As all cases lead to a contradiction, the assumption that $p$ is not interior to $V_i$ must be false.
\end{proof}
\fi

\iffull
\subsection{Intersections of non-degenerate cells}
\label{sec:intersection} 

In \cref{sec:degenerate}, we studied the shapes of individual cells of additively weighted Voronoi diagrams. In this section we instead consider the intersection patterns of cells.

\begin{figure}
\centering\includegraphics[scale=0.42]{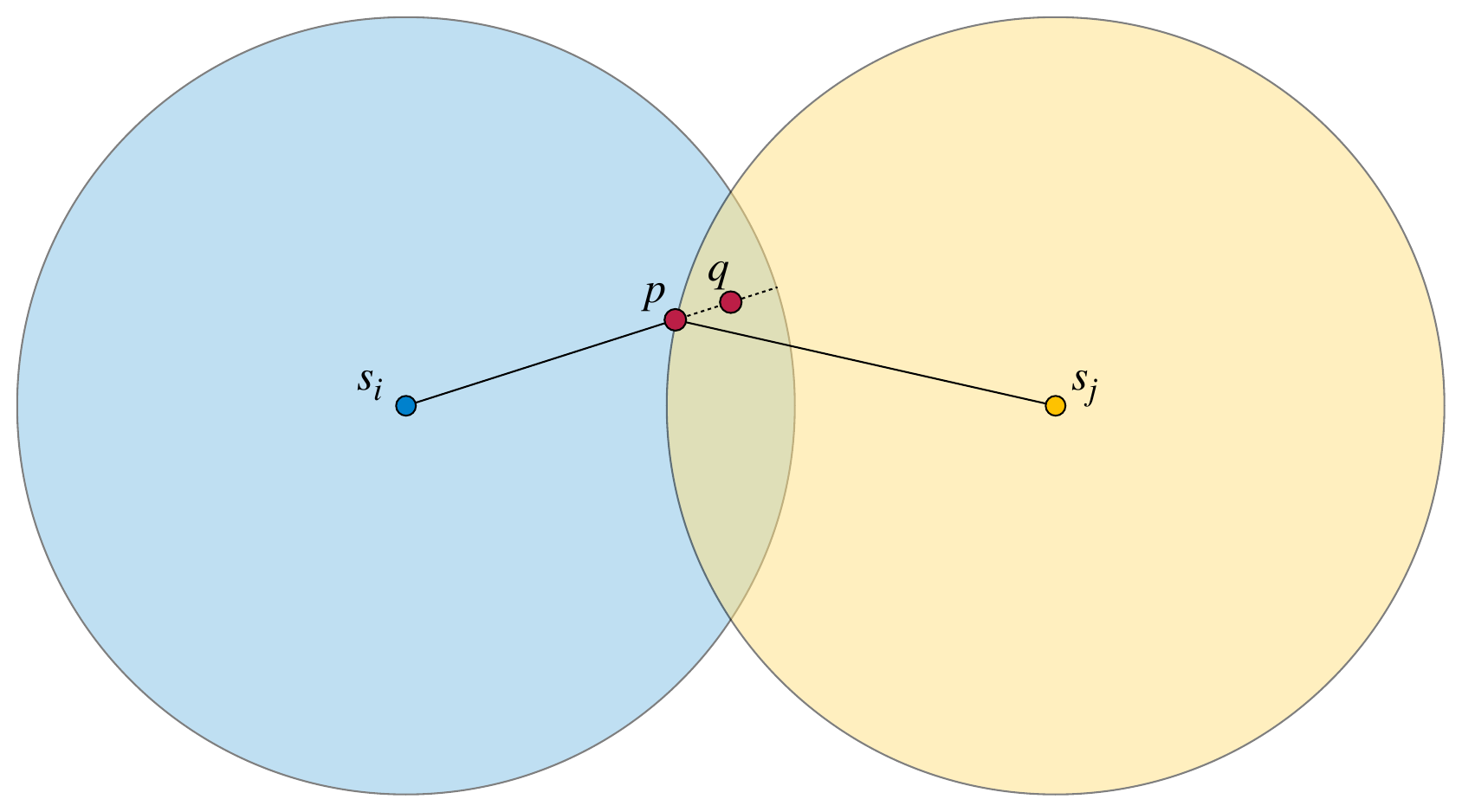}
\caption{Figure for the proof of \cref{lem:non-overlapping}: if a point $p$ could belong to the interior of the Voronoi cell $V_i$ for non-degenerate site $s_i$, and also belong to the Voronoi cell for $s_j$, then some point $q$ past $p$ on the same ray would also belong to $V_i$, contradicting the Lipschitz property of distances.}
\label{fig:non-overlapping}
\end{figure}

\begin{lemma}
\label{lem:non-overlapping}
Let $s_i$ and $s_j$ be non-degenerate sites of an additively weighted Voronoi diagram for a 2-dimensional strictly convex distance function. Then the interior of $V_i$ is disjoint from $V_j$, and vice versa.
\end{lemma}

\begin{proof}
Suppose for a contradiction that there existed a point $p$ interior to $V_i$ that also belonged to~$V_j$.
Then segments $ps_i$ and $ps_j$ must be disjoint except for their shared endpoint, for otherwise one would be a subset of the other, and one of the two sites would belong to the Voronoi cell of the other, contradicting the assumption that both are non-degenerate. Because~$p$ belongs to both $V_i$ and $V_j$, it has equal weighted distance to both sites.
Because~$p$ is interior to $V_i$, we can extend line segment $ps_i$ within a neighborhood of $p$ inside $V_i$, to obtain
a line segment $qs_i$ that still remains entirely interior to $V_i$ (\cref{fig:non-overlapping}). By the Lipschitz property of distances, $q$ is farther than $p$ from $s_i$ by an amount equal to the distance from $p$ to $q$, while the distance from $q$ to $s_j$ cannot increase as quickly.
Therefore, $q$ belongs to the Voronoi diagram of $V_j$ but not of $V_i$, contradicting our construction of $q$ as a point in~$V_i$. This contradiction shows that our assumption of a point interior to $V_i$ and in $V_j$ cannot hold: no such point exists. Symmetrically, no point interior to $V_j$ can also belong to $V_i$.
\end{proof}

\begin{lemma}
\label{lem:cross}
Let $s_i$ and $s_j$ be distinct non-degenerate sites of an additively weighted Voronoi diagram of a 2-dimensional strictly convex distance function, and let $p$ and $q$ be distinct points of $V_i$ and $V_j$ respectively. Then line segments (or arcs)  $ps_i$ and $qs_j$ are disjoint.
\end{lemma}

\begin{proof}
Otherwise, one of these two segments would contain a point in the relative interior of the other segment or arc. By \cref{lem:interior} this point would be interior to its Voronoi cell, contradicting \cref{lem:non-overlapping}.
\end{proof}

Finally, we consider intersections of three Voronoi cells, in the main lemma needed for our proof of our non-Euclidean Erd\H{o}s--Anning theorem:

\begin{figure}[t]
\centering\includegraphics[scale=0.21]{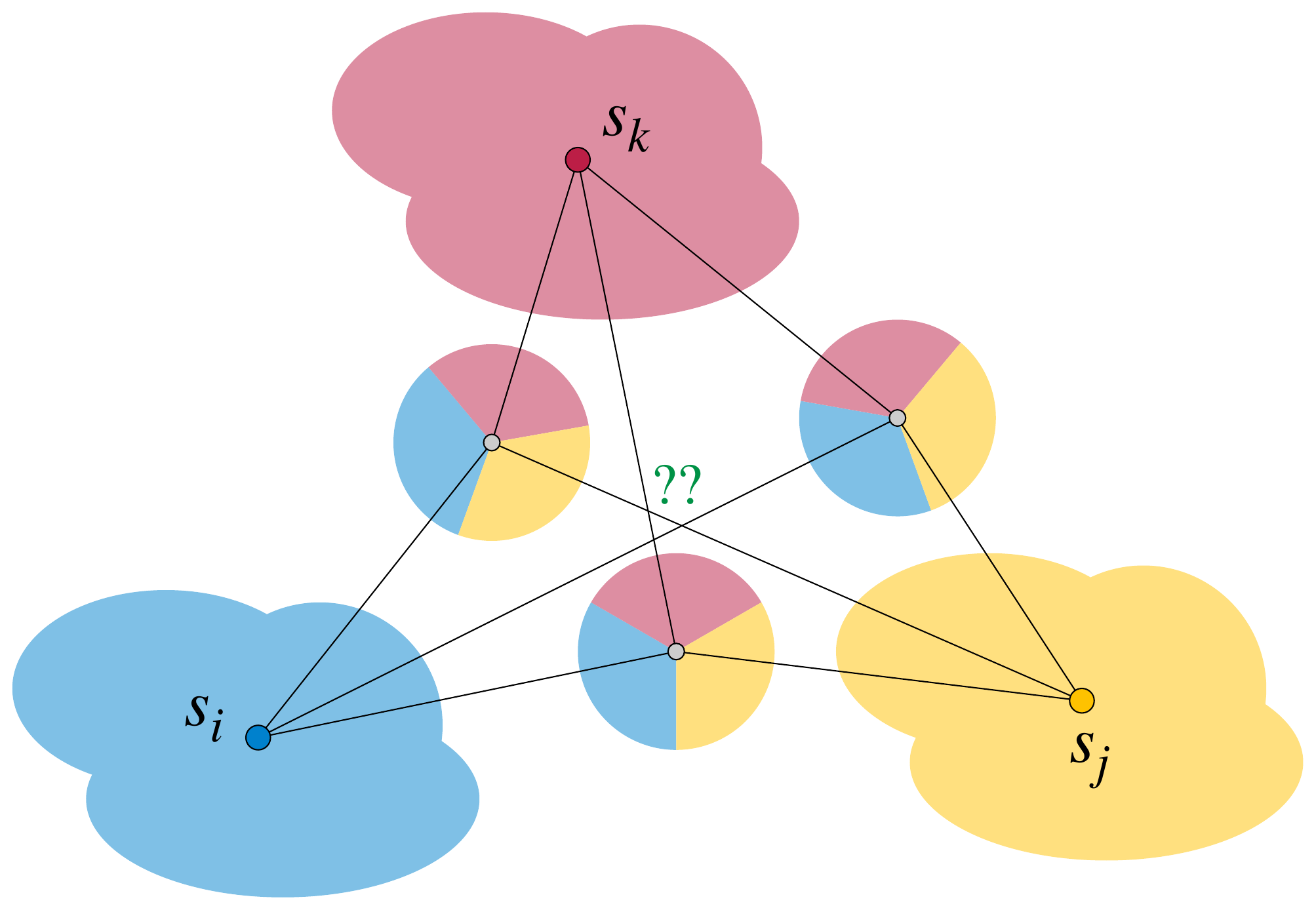}
\caption{Figure for the proof of \cref{lem:two}: Three non-degenerate sites $s_i$, $s_j$, and $s_k$, and three points belonging to the intersection of Voronoi cells $V_i\cap V_j\cap V_k$, would form a drawing of the non-planar graph $K_{3,3}$.}
\label{fig:utility}
\end{figure}

\begin{lemma}
\label{lem:two}
Let $s_i$, $s_j$, and $s_k$ be non-collinear sites of an additively weighted Voronoi diagram of a 2-dimensional strictly convex distance function. Then $V_i\cap V_j\cap V_k$ consists of at most two points.
\end{lemma}

\begin{proof}
Without loss of generality we may consider that these are the only three sites of the Voronoi diagram, because adding more sites can only reduce the points in $V_i\cap V_j\cap V_k$ by causing some of those points to belong to Voronoi cells of other sites. By \cref{lem:interior} and \cref{lem:non-overlapping}, the whole space is covered by the cells of non-degenerate sites, with disjoint interiors. We distinguish cases by how many of the three given sites are non-degenerate:
\begin{itemize}
\item Suppose that there is only one non-degenerate site, without loss of generality $s_i$. Then, $V_i$ is the whole space, and $V_j$ and $V_k$ lie on rays extending away from $s_i$. These rays diverge from each other (by non-collinearity) and must be disjoint, so $|V_i\cap V_j\cap V_k|=0$.
\item In the next case, there are two non-degenerate sites, without loss of generality $s_i$ and $s_j$. Because it is degenerate, $s_k$ belongs to at least one of their Voronoi cells, without loss of generality $V_i$. By \cref{lem:ray}, $V_k$ has the form $R\cap V_i$, where $R$ is a ray collinear with $s_i$. By \cref{lem:interior}, only the farthest point from $s_i$ on $R\cap V_i$ can be a boundary point of $V_i$; all the other points are interior to $V_i$. By \cref{lem:non-overlapping}, only the one boundary point can also belong to $V_j$. Thus, in this case, $|V_i\cap V_j\cap V_k|\le 1$ again.

\item In the remaining case, all three sites are non-degenerate. In particular, none of them belong to $V_i\cap V_j\cap V_k$. Consider the system of line segments connecting each site to each point of $V_i\cap V_j\cap V_k$. By \cref{lem:cross} these line segments cannot intersect except at shared endpoints, so they form a planar drawing of a complete bipartite graph with the three sites on one side of its bipartition and the points of $V_i\cap V_j\cap V_k$ on the other side. If there could be three points in $V_i\cap V_j\cap V_k$, this would give a planar drawing of $K_{3,3}$, a non-planar graph (\cref{fig:utility}). Therefore, $|V_i\cap V_j\cap V_k|\le 2$.\qedhere
\end{itemize}
\end{proof}
\fi

\subsection{Erd\H{o}s--Anning theorems for strict convex distance functions}

\begin{theorem}
\label{thm:triangle}
Let $s_1$, $s_2$, and $s_3$ be three non-collinear points of a 2-dimensional strictly convex distance function $d$. Then the number of points that can have integer distances to all three of $s_1$, $s_2$, and $s_3$ is $O\bigl(d(s_1,s_2)\cdot d(s_1,s_3)\bigr)$.
\end{theorem}

\begin{proof}
Let $p$ be a point with integer distances to all three of $s_1$, $s_2$, and $s_3$. Consider the additively weighted Voronoi diagram of these three sites, with weights $w_i=d(s_1,p)-d(s_i,p)$. Then at $p$, the three additively weighted distances all equal $d(s_1,p)$, so all three sites are equal nearest neighbors. This means that $p$ must belong to the triple intersection of Voronoi cells $V_1\cap V_2\cap V_3$.

Weight $w_1$ is identically zero, and by the triangle inequality the remaining two weights have absolute values $|d(s_1,p)-d(s_2,p)|$ and $|d(s_1,p)-d(s_3,p)|$ that are at most $d(s_1,s_2)$ and $d(s_1,s_3)$ respectively. Therefore, there are at most $\bigl(2d(s_1,s_2)+1\bigr)\cdot \bigl(2(d(s_1,s_3)+1\bigr)$ combinations of weights defining an additively weighted Voronoi diagram for which $p$ might be a triple intersection point, and at most two triple intersection points per diagram by \cref{lem:two}, giving a total of at most $2\bigl(2d(s_1,s_2)+1\bigr)\cdot \bigl(2(d(s_1,s_3)+1\bigr)$ possible points that can have integer distances to all three of $s_1$, $s_2$, and $s_3$.
\end{proof}

The following generalization of the Erd\H{o}s--Anning theorem follows immediately:

\begin{theorem}
\label{thm:finite}
If a set $S$ of points of a 2-dimensional strictly convex distance function has the property that all distances between points of $S$ are integers, then $S$ is finite or collinear.
\end{theorem}

\begin{proof}
If $S$ is not collinear, it has three non-collinear points to which we apply \cref{thm:triangle}.
\end{proof}

\subsection{Diameter bounds}

The number of points in a set with integer distances can also be bounded linearly in the diameter of the set, generalizing a result of Solymosi~\cite{Sol-DCG-03} for Euclidean distance.

\begin{theorem}
\label{thm:dist-fun-diam}
If a set $S$ of points of a 2-dimensional strictly convex distance function has diameter $D$ (measured according to that distance function), and all distances in $S$ are integers, then $|S|=O(D)$, with a constant of proportionality that depends on the distance function.
\end{theorem}

\begin{proof}
We may assume without loss of generality that the unit ball of the given distance function~$d$ is contained in a Euclidean unit circle. For, if it does not, let $\xi$ be the maximum Euclidean distance from the origin to any point of the unit $d$-ball, and scale both $S$ and $d$ by $1/\xi$, leaving distances between the scaled points unchanged and allowing the $d$-ball to fit within a Euclidean unit circle. Let $\rho$ be the minimum Euclidean distance from the origin of the unit circle of~$d$; note that, if we treat the distance function as fixed, then~$\rho$ is a constant. Then the function measured by~$d$ is lower bounded by the Euclidean distance, and upper bounded by $\tfrac{1}{\rho}$ times the Euclidean distance. From the lower bound on $d$, the Euclidean diameter of $S$ is at most its $d$-diameter, $D$.

\begin{figure}[t]
\centering\includegraphics[scale=0.21]{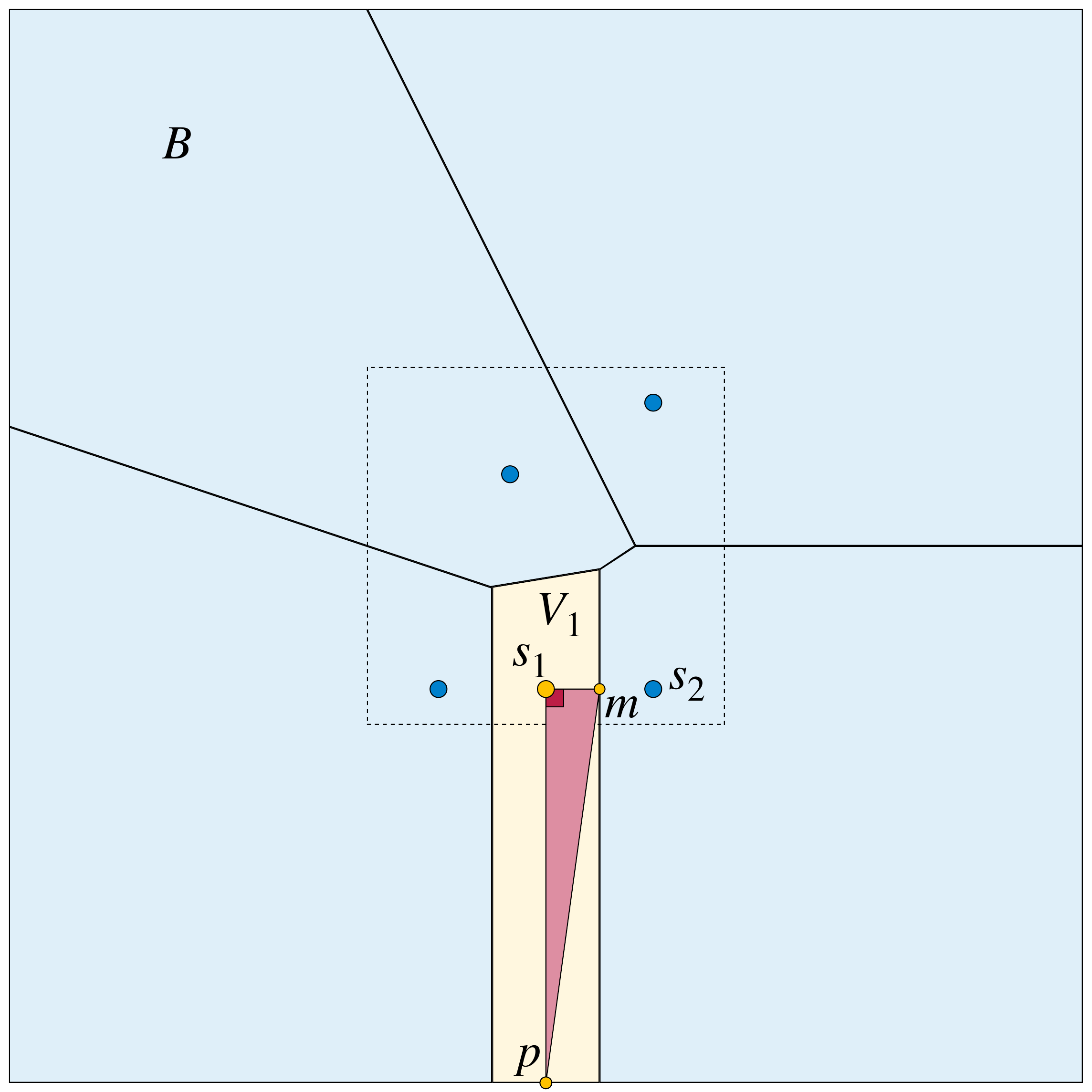}
\caption{Figure for the proof of \cref{thm:dist-fun-diam}: After enclosing the given points in a bounding box $B$ of side length three times the diameter $D$, $s_1$ is the point whose Euclidean Voronoi cell $V_1$ has the smallest intersection with $B$, $s_2$ is its nearest neighbor, $m$ is their midpoint, and right triangle $ps_1m$ is chosen to have as large an area as possible within $V_1\cap B$. In the case shown, $p$ is on the boundary of $B$, at distance greater than $D$ from $s_1$, from which it follows that $d(s_1,s_2)=O(D/n)$.}
\label{fig:small-triangle}
\end{figure}

Enclose $S$ in a square bounding box of (Euclidean) side length $D$, and center this within a larger square $B$ of side length $3D$, so that all points of $S$ are at Euclidean distance at least~$D$ from the boundary of~$B$. Construct the (Euclidean) Voronoi diagram of $S$, and let $s_1$ be the point of $S$ whose Voronoi cell $V_1$, intersected with $B$, has the smallest Euclidean area, $\le 9D^2/n=O(D^2/n)$.
Let $s_2$ be the Euclidean nearest neighbor of $s_1$ in $S$, let $m$ be the midpoint of $s_1$ and $s_2$, and construct a right triangle $ps_1m$ with right angle at $s_1$, contained in $V_1\cap B$, with the largest possible Euclidean area. (See \cref{fig:small-triangle} for an example.) We distinguish two cases:
\begin{itemize}
\item If $p$ is on the boundary of $B$, as shown in \cref{fig:small-triangle}, it is at Euclidean distance $\ge D$ from~$s_1$. In order for right triangle $ps_1s_2$ to have area $O(D^2/n)$, at most the area of the set $V_1\cap B$ that contains it, edge $s_1s_2$ must have Euclidean length $O(D/n)$, and therefore also $d(s_1,s_2)=O(D/n)$ (larger than the Euclidean length by at most the constant  factor~$\tfrac{1}{\rho}$). For this distance to be a nonzero integer, $n$ must be $O(D)$.
\item Otherwise, $p$ is on the boundary of $V_1$, and equidistant (in Euclidean distance) from $s_1$ and some site $s_3$ which cannot be collinear with $s_1s_2$. Therefore, the Euclidean distance from $s_1$ to $p$ is at least half the Euclidean distance from $s_1$ to $s_3$. Combining this inequality with the area of the triangle and the $\tfrac{1}{\rho}$ factor by which $d$ can increase over Euclidean distance gives the bound
$d(s_1,s_2)\cdot d(s_1,s_3)=O(D^2/n)$. It follows from \cref{thm:triangle} that $n=O(D^2/n)$ and therefore that $n=O(D)$.\qedhere
\end{itemize}
\end{proof}

The examples of $n\times n$ integer grids for the $L_1$ or $L_\infty$ distance show that the assumption of strict convexity is necessary for this result.

\section{Riemannian manifolds}
\label{sec:riemann}

In this section we prove a form of the Erd\H{o}s--Anning theorem for two-dimensional complete Riemannian manifolds. These can be defined intrinsically (without respect to an embedding) as smooth manifolds equipped with a smoothly varying inner product structure on their tangent space. ``Completeness'' in this context means that all Cauchy sequences converge; for instance, puncturing the Euclidean plane by removing the origin would produce a space that is not complete.
\iffull
By the Nash embedding theorem~\cite{Nas-AM-56} it is equivalent (and may be easier) to think of these as two-dimensional surfaces, smoothly embedded into a Euclidean space of bounded dimension, with distance measured by the Euclidean length of curves on the surface. Standard examples
\fi
of such spaces include the Euclidean plane itself, the surface of any smooth convex body in three dimensions such as a sphere, and the hyperbolic plane.

One complication in these spaces is that their usual analogues of straight lines, \emph{geodesics}, are only locally shortest curves between their points: two or more points may belong to a common geodesic that does not contain the shortest curve between them. Indeed, in some cases a geodesic may cross itself or form a dense subset of its surface. For this reason, we will generally work with shortest curves (between some two points of the space) rather than geodesics. These curves may not be unique: for instance, for antipodal points on a sphere, there are infinitely many shortest curves, covering the whole sphere. By the Hopf--Rinow theorem~\cite{HopRin-CMH-31}, every complete Riemannian manifold has a shortest curve between every two points.  Thus, these are geodesic metric 
\iffull
spaces, and in particular the Lipschitz property of distances (\cref{obs:lipschitz}) applies to them. 
\fi

We will not assume that the surfaces we consider are orientable. One way to parameterize a two-dimensional surface, without assuming orientability, is by its \emph{Euler genus}, the number $2-\chi$ where $\chi$ is the Euler characteristic. It is possible for a Riemannian manifold to have infinite genus, but (except for \cref{ex:high-genus}, which motivates this limitation) we will only work with surfaces of bounded genus.

\subsection{Examples}

\iffull
We have the following analogue of \cref{ex:complete}:
\fi

\begin{example}
Consider any finite set of three or more points $S$ on a unit sphere, no three of which belong to a great circle. Then each two points of $S$ have a unique shortest curve which does not pass through other points of $S$. Perturb the sphere by placing a small smooth bump along each shortest curve, increasing the distance between its two points to a rational number without changing the other distances, and then scale the result to clear the denominators of these rational distances. The result is a complete Riemannian manifold of Euler genus zero with an arbitrary finite number of points at integer distances from each other.
\end{example}

It is necessary to assume that our Riemannian manifolds are complete, because of the following example:

\begin{example}
\label{ex:infinite-cone}
Consider the points coordinatized by two real numbers $r$ and $\theta$, with $r$ positive and $\theta$ arbitrary, and with the standard locally Euclidean geometry of the polar coordinate system. This can be thought of as a cone with an infinite angle at its removed apex~\cite{WebHofWol-PNAS-05}. Its Euler genus is zero (topologically it is just an open halfplane). When two of its points have $\theta$-coordinates that differ by at least $\pi$, their distance equals the sum of their $r$-coordinates (distances from the removed apex). Therefore, the points $\{(\tfrac12,i\pi)\mid i\in\mathbb{Z}\}$ form an infinite set at unit distance from each other. They have no shortest curves, because any such curve would pass through the missing apex point.
\end{example}

\begin{figure}[t]
\centering\includegraphics[width=0.35\textwidth]{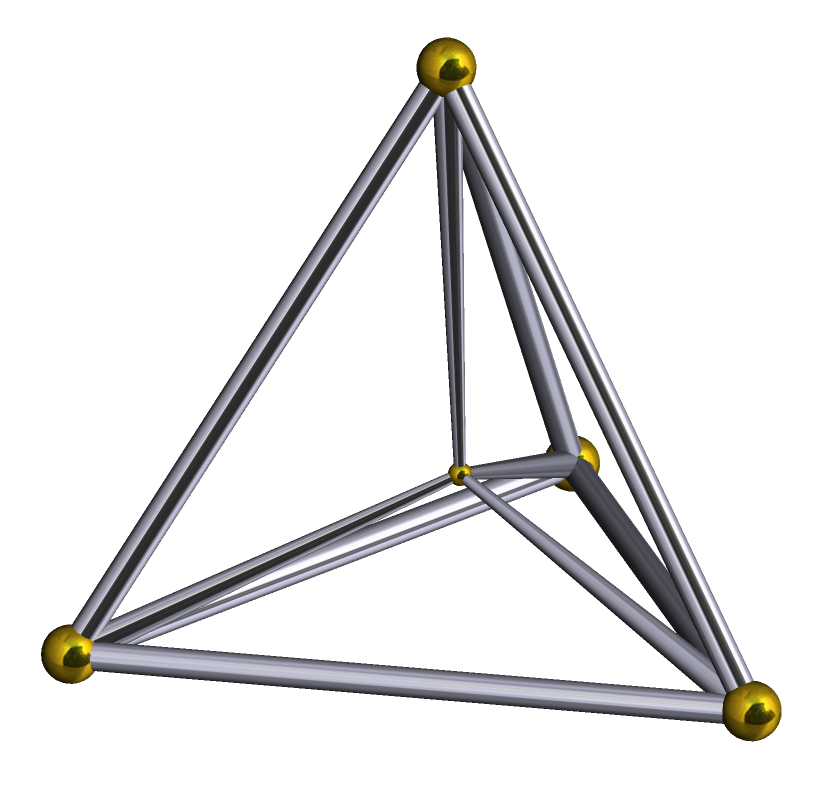}
\caption{Five small spheres connected in pairs by cylindrical tubes. Smoothing the sphere-cube edges produces a complete Riemannian manifold. Public domain image from \url{https://commons.wikimedia.org/wiki/File:Schlegel_wireframe_5-cell.png}.}
\label{fig:ball-and-stick}
\end{figure}

As the next example shows, to obtain a generalization of the Erdos--Anning theorem to complete Riemannian manifolds, we will also need to assume bounded genus, and our bounds on numbers of points in integer distance sets will need to depend on the genus.

\begin{example}
\label{ex:high-genus}
For any $n$, one may form a complete Riemannian manifold from $n$ spheres, connected by cutting out disks from pairs of spheres, attaching a cylindrical tube between the boundaries of the two disks, and smoothing the parts of the surface where the disks meet the tubes. This construction may be performed abstractly as a topological space; it is not necessary to embed the spheres and tubes into three-dimensional space. Let $S$ consist of one point per sphere, for spheres small enough that all distances within one sphere are less than~$\tfrac13$; then by adjusting the lengths of all the tubes, all distances between points of $S$ can be made equal to one. This construction produces arbitrarily large finite sets of points at unit distance from each other, with genus~$O(n^2)$. Because they all have the same distance, no three of these points belong to a shortest curve.

The connected sum of an infinite sequence of examples of this type for $n=4,5,6,\dots$, using tubes to connect the example for each $n$ to the examples for $n-1$ and $n+1$, forms a single complete Riemannian manifold of infinite genus in which one can find arbitrarily large (but nevertheless finite) sets of points that are all at unit distance from each other.

Although the topological structure of this example is complicated, it can be embedded in a distance-preserving way into a Riemannian metric on $\mathbb{R}^3$. For each $n$, choose the spheres and tubes of the $n$-point example to have a small enough radius that they can all be smoothly embedded within a unit ball of $\mathbb{R}^3$, coiling the tubes to cause them to have the desired lengths. Form the connected sum of these examples as before and embed it smoothly in $\mathbb{R}^3$ using disjoint unit balls for each term in the connected sum. Thicken the resulting smoothly embedded surface $S$ by a small amount (varying smoothly over the surface) chosen to be small enough that the thickened set remains smoothly embedded in $\mathbb{R}^3$, and give the resulting thickened set a Riemannian metric that agrees with the metric on $S$ (and the Euclidean distance of $\mathbb{R}^3$) in directions parallel to $S$, but that is longer in the direction perpendicular to $S$ by a factor that varies smoothly from very large for points on $S$ to one at the boundary of this thickened set. Then, for the remaining points of~$\mathbb{R}^3$, outside the thickened set, use the usual Euclidean distance. By choosing this perpendicular lengthening factor large enough, it is possible to ensure that escaping from $S$ to the boundary of the thickened set takes more than unit distance, and that distances within $S$ are preserved.
\end{example}

Richard Guy asked whether the \emph{equilateral dimension} of Riemannian manifolds, the maximum cardinality of a set of points in the manifold for which all distances are one, can be bounded by a function of the dimension of the manifold~\cite{Guy-AMM-83}. A positive answer was known only for the case of complete Riemannian manifolds of bounded curvature~\cite{Man-RH-13}. \cref{ex:infinite-cone} provides a counterexample for locally Euclidean and topologically simple but incomplete Riemannian 2-manifolds. \cref{ex:high-genus} provides a counterexample for complete Riemannian 2-manifolds of unbounded genus and for topologically simple complete Riemannian 3-manifolds. On the other hand, our generalization of the Erd\H{o}s--Anning theorem to complete Riemannian 2-manifolds of bounded genus also provides a bound on the equilateral dimension of these manifolds, as a function of their genus (\cref{cor:equilateral-dimension}).

Our proof of the Erd\H{o}s--Anning theorem for strictly convex distance functions relied on the non-planarity of $K_{3,3}$. For complete Riemannian manifolds of bounded genus we need to reach further for graphs that cannot be embedded, as the following example shows.

\iffull
\begin{figure}[t]
\centering\includegraphics[width=0.5\textwidth]{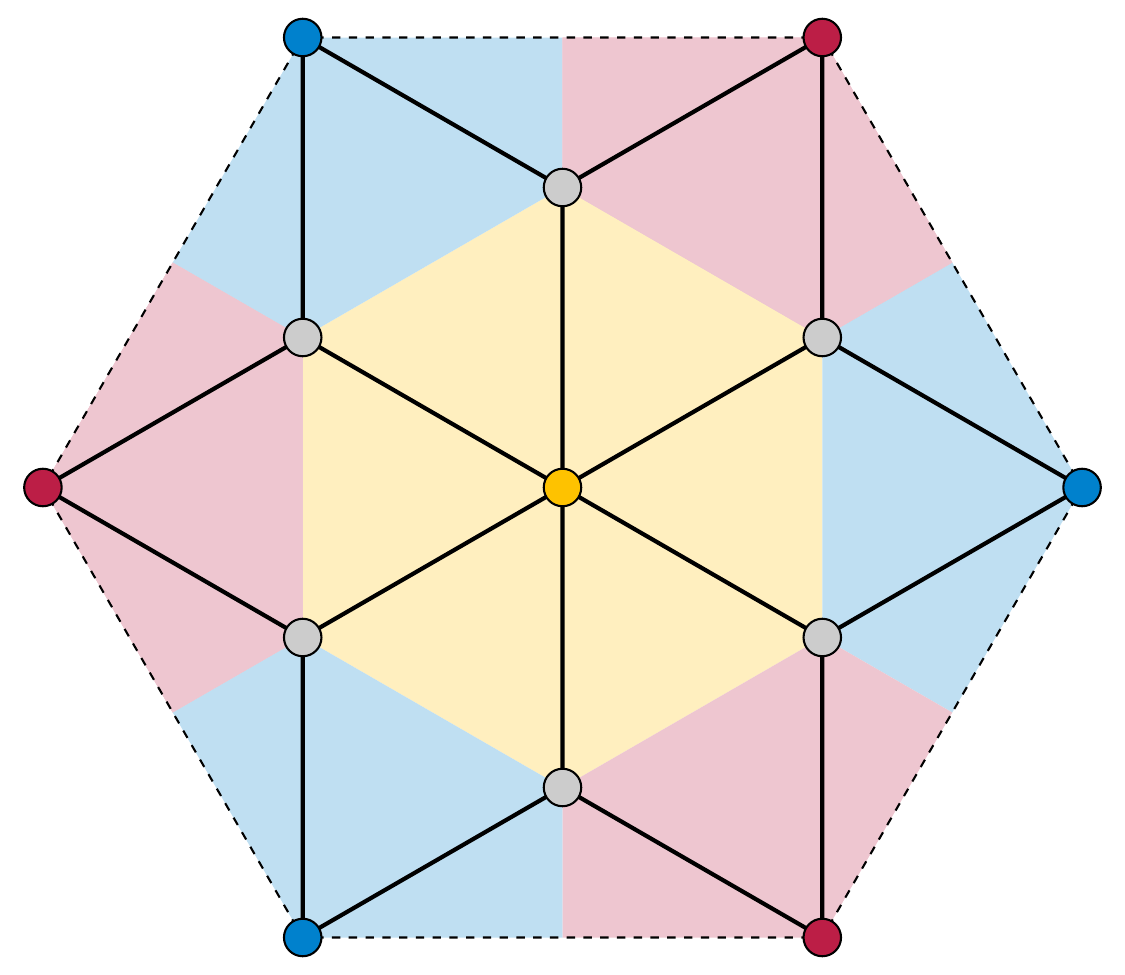}
\caption{An embedding of $K_{3,6}$ with shortest-curve edges on a hexagonal torus (\cref{ex:k36-hex-torus}). The colors indicate the Voronoi diagram of the points on the three-vertex side of the bipartition.}
\label{fig:k36-hex-torus}
\end{figure}

\begin{example}
\label{ex:k36-hex-torus}
Form a complete Riemannian manifold (the \emph{hexagonal torus}) by gluing opposite edges of a Euclidean regular hexagon.  The hexagon's six vertices become two points of the torus (red and blue in \cref{fig:k36-hex-torus}). Let $S$ consist of these two points and the center of the hexagon (yellow). Six points (gray) are equidistant from these three points. Connecting the three points of $S$ by shortest curves to these six points embeds $K_{3,6}$ on the torus.
\end{example}
\fi

\iffull
\subsection{Voronoi cells}
\label{sec:riemann-vor1} 

Our definitions of Voronoi diagrams for finite sets of sites, additively weighted Voronoi diagrams, and Voronoi cells apply to any metric space. We may define kernels and star-shaped sets, in complete Riemannian manifolds, by analogy to their Euclidean definitions. We define the \emph{kernel} of a set $S$ to be the set of points $p$ such that, for every point $q\in S$, every shortest curve between $p$ and $q$ belongs to $S$. We define $S$ to be \emph{star-shaped} if its kernel is non-empty. Then we have the following analogue of \cref{lem:star}. The proof, which depends only on the space being geodesic and on the Lipschitz property of distances, is unchanged.

\begin{lemma}
\label{lem:Riemann-star}
For additively weighted Voronoi diagrams of complete Riemannian manifolds, every non-empty Voronoi cell is star-shaped with its site in its kernel.
\end{lemma}

As before, we define a \emph{degenerate site} to be a site that belongs to the Voronoi cell of another site, and a \emph{non-degenerate site} to be a site that is not degenerate. To formulate an analogue of \cref{lem:ray}, we also need a definition of a ray in a complete Riemannian manifold. To do this, we use the \emph{exponential map}. Given a point $p$ of the manifold, and a vector $v$ in the tangent space at $p$, the exponential map takes $v$ to $\exp_p v$, another point in the manifold obtained by following a geodesic from $p$ for a length and direction specified by $v$. For Riemannian manifolds in general, this is defined only for a neighborhood of the origin in the tangent space, but for complete Riemannian manifolds it is defined for the whole tangent space. For any point $p$ of the manifold, there is a number $r>0$, the \emph{injectivity radius}, such that the exponential map is an injection on vectors in the tangent space at $p$ with radius~$<r$, and not an injection for vectors with radius $\le r$. The exponential map is a smooth topological equivalence (a diffeomorphism) between the disk of vectors of radius~$<r$ and its image, a neighborhood of $p$ in the manifold. We define a \emph{ray} in a complete Riemannian manifold to be a set of the form
\[
\left\{ \exp_p cv \mid 0 < c \right\}
\]
for a fixed choice of $p$ (the starting point or \emph{apex} of the ray) and $v$ (a vector specifying the direction of the ray), and for all positive scalars $c$.

Another complication here is that, for convex distance functions, the intersection of a ray with a star-shaped set (having its apex in the kernel) is a line segment, and therefore a shortest curve. For  complete Riemannian manifolds, a ray may continue through a star-shaped set past the last point for which it forms a shortest curve, or may even have a disconnected intersection with the set. It may also pass through some points of the manifold more than once.  In a Voronoi diagram, we define the \emph{Voronoi segment} of a ray, having one of its sites~$s_i$ as an apex, to be the union of all shortest curves from points in $V_i$ to $s_i$ that lie along the same ray.

\begin{lemma}
\label{lem:Riemann-ray}
Let $s_i$ be a degenerate site of an additively weighted Voronoi diagram of a complete Riemannian manifold, with $s_i$ belonging to the Voronoi cell $V_j$ of another site $s_j$. Let $R_j$ be a ray with apex $s_j$ containing a shortest curve from $s_i$ to $s_j$, and let $R_i$ be a ray with apex $s_i$ continuing in the same direction, away from $s_i$. Then either $s_i$ is empty, or $V_i$ is the intersection of $R_i$ with the Voronoi segment of $R_j$.
\end{lemma}

\begin{proof}
As in the proof of \cref{lem:ray}, if the weighted additive distance from $s_i$ to $s_j$ is less than the weighted distance from $s_i$ to itself, then $s_i$ does not belong to its own Voronoi cell and by \cref{lem:Riemann-star} the cell must be empty. Otherwise, at $s_i$, the weighted distances to $s_i$ and~$s_j$ are equal. Because the lengths of curves along rays increase at the same rate as distance traveled along a ray, the weighted lengths of curves along $R_i$ and $R_j$ remain equal for all points of $R_i$. At every point of the Voronoi segment, this weighted distance is the minimum distance to a site, and at every point of the past the endpoint of the Voronoi segment, it is not.
\end{proof}

By repeatedly using \cref{lem:Riemann-ray} to replace the Voronoi cells of degenerate sites with supersets that are also Voronoi cells, until no more such replacements can be performed, we obtain an analogue of \cref{cor:non-degenerate-cover}:

\begin{corollary}
In an additively weighted Voronoi diagram of a complete Riemannian manifold, every point belongs to the Voronoi cell of at least one non-degenerate site.
\end{corollary}

We also have the following analogue of \cref{lem:non-degenerate-interior}:

\begin{lemma}
\label{lem:Riemann-ndgi}
A site $s_i$ of an additively weighted Voronoi diagram of a 2-dimensional complete Riemannian manifold is non-degenerate if and only if it is an interior point of its cell $V_i$.
\end{lemma}

The proof is the same as that for \cref{lem:non-degenerate-interior}.

We interpreted \cref{lem:interior} as stating that each ray with a non-degenerate site as its initial point passes through at most one boundary point of its Voronoi cell. This is not true for complete Riemannian manifolds, because rays can re-enter the Voronoi cell after the end of the Voronoi segment of the ray. On the other hand, a Voronoi segment of bounded length may have its endpoint in the interior of a cell, and not pass through any boundary points of the cell. Nevertheless, we have the following analogue of \cref{lem:interior}:

\begin{lemma}
\label{lem:Riemann-interior}
Let $s_i$ be a non-degenerate site of an additively weighted Voronoi diagram of a 2-dimensional complete Riemannian manifold, let $L$ be a Voronoi segment in $V_i$ having $s_i$ as an endpoint, and let $p$ be a point in the relative interior of $L$. Then $p$ is interior to $V_i$.
\end{lemma}

\begin{proof}
As in the proof of \cref{lem:interior}, the only other possibility is that there is some other Voronoi cell $V_j$ also containing $p$, such that every neighborhood of $p$ contains points of $V_j\setminus V_j$. Then the shortest curves from $p$ to $s_i$ and to $s_j$ cannot follow the same ray from~$p$; otherwise,~$s_i$ would be degenerate (contradicting the assumption in the statement of the lemma) or $s_j$ would be degenerate (contradicting the assumption that $V_j\setminus V_i$ is nonempty). For any point~$q$ past~$p$ on~$L$, there would exist curves of equal length from~$q$ to~$s_i$ along~$L$, and from~$q$ to~$s_j$ along~$L$ t~ $p$ and then from~$p$ to~$s_j$; but because the curve to $s_j$ turns at $p$ it cannot be a shortest curve as it can be shortcut near $p$, so $q$ is closer to $s_j$ than $s_i$. But this would contradict the definition of $L$ as a Voronoi segment.
\end{proof}
\fi

\iffull
\subsection{Intersections of non-degenerate cells}
\label{sec:riemann-vor2} 

Continuing the analogy from convex distance functions to Riemannian manifolds, we have the following analogue of \cref{lem:non-overlapping}. The proof is somewhat different because of the possibility that a point interior to a Voronoi cell may nevertheless be an endpoint of a Voronoi segment.

\begin{lemma}
\label{lem:Riemann-non-overlapping}
Let $s_i$ and $s_j$ be non-degenerate sites of an additively weighted Voronoi diagram for a 2-dimensional complete Riemannian manifold. Then the interior of $V_i$ is disjoint from~$V_j$, and vice versa.
\end{lemma}

\begin{proof}
Suppose for a contradiction that there existed a point $p$ interior to $V_i$ that also belonged to $V_j$. Then any shortest curves $ps_i$ and $ps_j$ must be disjoint except for their shared endpoint, for otherwise one would be a subset of the other, and one of the two sites would belong to the Voronoi cell of the other, contradicting the assumption that both are non-degenerate. Because $p$ is interior to $V_i$, there must be some point $q\ne p$ on segment $ps_j$ that also belongs to $V_i$. But by the Lipschitz property of distances, the weighted distance to to $s_j$ decreases more quickly along curve $pq$ than the weighted distance to $s_i$, a contradiction.
\end{proof}

Continuing the analogy, we have:

\begin{lemma}
\label{lem:Riemann-disjoint-segments}
Let $s_i$ and $s_j$ be distinct non-degenerate sites of an additively weighted Voronoi diagram of a 2-dimensional complete Riemannian manifold, and let $p$ and $q$ be distinct points of $V_i$ and $V_j$ respectively. Then line segments (or arcs)  $ps_i$ and $qs_j$ are disjoint.
\end{lemma}

\begin{proof}
Otherwise, one of these two segments would contain a point in the relative interior of the other segment or arc. By \cref{lem:Riemann-interior} this point would be interior to its Voronoi cell, contradicting \cref{lem:Riemann-non-overlapping}.
\end{proof}

To formulate an analogue of \cref{lem:two} we must first account for examples like \cref{ex:k36-hex-torus} in which triples of Voronoi cells have more than two points of intersection, corresponding to drawings of graphs $K_{3,a}$ with $a>2$ and with all edges drawn as shortest curves.

\begin{lemma}
\label{lem:genus-bipartite}
On a surface of Euler genus $g$, the maximum number $a$ such that $K_{3,a}$ can be drawn without crossings on the surface is at most $2g+2$.
\end{lemma}

\begin{proof}
In any drawing of $K_{3,a}$ on any surface, all faces of the drawing have at least four sides, and each edge forms exactly two sides of faces (not necessarily distinct). Therefore, if $V$, $E$, and $F$ are the sets of vertices, edges, and faces of the drawing, we have  
\[
2|E|\ge 4|F|.\]
We also have
\[
|V|-|E|+|F|\ge 2-g
\]
by the definition of the Euler characteristic, where the inequality comes from the fact that we are not assuming that the faces of the drawing are disks. Using the first inequality to eliminate $|F|$ from the second gives
\[
|V|-\frac12 |E|\ge 2-g.
\]
or equivalently
\[
|E|\le 2|V| + 2g - 4.
\]
The graph $K_{3,a}$ has $|V|=3+a$ and $|E|=3a$. In order to have $3a\le 2(3+a)+2g-4$,
we must have $a\le 2g+2$.
\end{proof}

This is a known result; see, e.g., Bouchet~\cite{Bou-JCTB-78}, who credits the complete characterization of the genus of complete bipartite graphs to several works of Ringel from the mid-1960s. The Euler genus zero case of this lemma is the fact that $K_{3,3}$ is non-planar, while the Euler genus two case includes the fact that $K_{3,7}$ cannot be drawn on a torus. In these cases, the lemma is tight: $K_{3,2}$ is planar and \cref{ex:k36-hex-torus} shows that $K_{3,6}$ is toroidal.

Finally, this allows us to bound the number of triple intersections of Voronoi cells:

\begin{lemma}
\label{lem:Riemann-triples}
Let $s_i$, $s_j$, and $s_k$ be three sites of an additively weighted Voronoi diagram of a 2-dimensional complete Riemannian manifold of Euler genus $g$, with none of the three sites belonging to any shortest curve between the other two sites. Then $V_i\cap V_j\cap V_k$ consists of at most $2g+2$ points.
\end{lemma}

\begin{proof}
Without loss of generality we may consider that these are the only three sites of the Voronoi diagram, because adding more sites can only reduce the points in $V_i\cap V_j\cap V_k$ by causing some of those points to belong to Voronoi cells of other sites. By \cref{lem:Riemann-interior} and \cref{lem:Riemann-non-overlapping}, the whole space is covered by the cells of non-degenerate sites, with disjoint interiors. We distinguish cases by how many of the three given sites are non-degenerate:
\begin{itemize}
\item Suppose that there is only one non-degenerate site, without loss of generality $s_i$. Then, $V_i$ is the whole space, and $V_j$ and $V_k$ lie on Voronoi segments of $s_i$. These segments are distinct (else one of $s_j$ and $s_k$ would lie on a shortest curve from $s_i$ to the other). A point~$p$ interior to one Voronoi segment cannot belong to another Voronoi segment, because if it did the points past $p$ on the first segment would have an equally short curve to $s_i$ that follows the first segment to $p$ and then the second curve to $s_i$, but such a non-geodesic curve cannot be a shortest curve (it could be made shorter by adding a small shortcut near $p$). Therefore, in this case the only point of intersection can be a shared endpoint of the Voronoi segments of $s_i$ that contain $s_j$ and $s_k$, giving $|V_i\cap V_j\cap V_k|\le 1$.
\item In the next case, there are two non-degenerate sites, without loss of generality $s_i$ and $s_j$. Because it is degenerate, $s_k$ belongs to at least one of their Voronoi cells, without loss of generality $V_i$. By \cref{lem:ray}, $V_k$ is the intersection of a ray from $s_k$ with a Voronoi segment of $V_i$. By \cref{lem:Riemann-interior} and \cref{lem:Riemann-non-overlapping}, only the endpoint of this Voronoi segment can belong to $V_j$. Thus, in this case, $|V_i\cap V_j\cap V_k|\le 1$ again.
\item In the remaining case, all three sites are non-degenerate. In particular, none of them belong to $V_i\cap V_j\cap V_k$. Consider the system of Voronoi segments connecting each site to each point of $V_i\cap V_j\cap V_k$. These line segments cannot intersect except at shared endpoints, so they form a planar drawing of a complete bipartite graph with the three sites on one side of its bipartition and the points of $V_i\cap V_j\cap V_k$ on the other side.  By \cref{lem:genus-bipartite}, $|V_i\cap V_j\cap V_k|\le 2g+2$.\qedhere
\end{itemize}
\end{proof}
\fi

\subsection{Erd\H{o}s--Anning theorems for Riemannian manifolds}

\begin{theorem}
\label{thm:Riemann-triangle}
Let $s_1$, $s_2$, and $s_3$ be three points of a 2-dimensional complete Riemannian manifold of Euler genus $g$, with distance function $d$. Suppose also that there is no shortest curve containing all three of these points. Then the number of points that can have integer distances to all three of $s_1$, $s_2$, and $s_3$ is $O\bigl((g+1)\cdot d(s_1,s_2)\cdot d(s_1,s_3)\bigr)$.
\end{theorem}

Except for the factor of $g$ coming from
\iffull
\cref{lem:Riemann-triples},
\fi
the proof is the same as for \cref{thm:triangle}.
For this theorem, the possibility that all triples of points lie on shortest curves (preventing the theorem from being applied) corresponds to the possibility of all points being collinear in the Euclidean plane or for convex distance functions. It is a strengthening of the more obvious generalization of collinearity, that all points of $S$ lie on a common geodesic:

\begin{lemma}
\label{lem:Riemann-geodesic}
Let $S$ be a set of three or more points of a complete Riemannian manifold such that each three points of $S$ belong to a shortest curve.
Then $S$ is a subset of a geodesic, and every shortest curve between pairs of points in $S$ lies along this geodesic.
\end{lemma}

\begin{proof}
Let $s_1$ be any of the points, and consider the system of rays containing shortest curves to all the other points. Suppose for a contradiction that two rays $R_2$ and~$R_3$ in this system of rays contain shortest curves from $s_1$ to points $s_2$ and $s_3$ in $S$,
and that these two rays are distinct and non-opposite. Then if $s_1$ were on a shortest path from $s_2$ to $s_3$, then a shortest curve from $s_2$ to $s_3$ would follow $R_2$ from $s_2$ to $s_1$ and then follow~$R_3$ from $s_1$ to $s_3$. This cannot happen, because a shortest curve that turns from one ray to another at $s_1$ could be shortcut by a different curve in a neighborhood of $s_1$. Alternatively, if $s_2$ were on a shortest curve from $s_1$ to $s_3$, then we could get a shortest curve from $s_1$ to $s_3$ that follows $R_2$ from $s_1$ to $s_2$, turns at $s_2$, and then continues along the $R_3$ from $s_2$ to $s_3$, again an impossibility. The case that $s_3$ is on a shortest curve from $s_1$ to $s_2$ is symmetric. So none of the three points can be the middle point of a shortest curve containing all three points. This contradiction to the assumption of the lemma implies that rays $R_2$ and $R_3$ cannot exist, and that the system of rays contains either a single ray or two opposite rays. In either case, these two rays, and therefore all points of $S$, lie on a single geodesic.

If some two points $s_2$ and $s_3$ had a shortest curve that did not lie along this geodesic, then we could apply this same argument to the rays at $s_2$ containing the shortest curve to $s_1$ (along the geodesic through all points of $S$) and the ray containing the shortest curve from $s_2$ to $s_3$ (not along this geodesic), obtaining the same contradiction. Therefore all shortest curves lie on the same geodesic.
\end{proof}

The conclusion of this lemma cannot be strengthened to state that each pair of points in $S$ has a unique shortest curve. For instance, consider an even number of points equally spaced on the equator of a sphere. In this example, each point has two shortest curves to its opposite point. A set of points whose shortest curves all lie on a single geodesic might still not have the property that each three points belong to a shortest curve: consider, for instance, three equally spaced points on the equator of a sphere.

Combining \cref{thm:Riemann-triangle} with \cref{lem:Riemann-geodesic} gives us the following generalization of the Erd\H{o}s--Anning theorem:

\begin{theorem}
\label{thm:Riemann-finite}
If a set $S$ of points of a 2-dimensional complete Riemannian manifold of bounded genus has the property that all distances between points of $S$ are integers, then either $S$ is finite or there is a single geodesic of the manifold containing all points of $S$ and connecting each pair of points by a shortest curve.
\end{theorem}

We also have the following diameter bound, weaker than in the Euclidean or convex distance function case:

\begin{theorem}
\label{thm:Riemann-diameter}
If a set $S$ of points of a 2-dimensional complete Riemannian manifold of bounded genus $g$ has integer distances and diameter $D$, then $|S|=O(gD^2)$.
\end{theorem}

\begin{proof}
If $S$ has three points not all on a single shortest curve, this follows from \cref{thm:Riemann-triangle}. Otherwise, the points all have shortest curves along a geodesic. Let $C$ be a shortest curve of length $D$ between some two points of $S$. If $C$ covers all points of $S$, there are at most $D+1$ points. Otherwise, let $s_1$ and $s_2$ be the two endpoints of $C$, and let $s_3$ be any point not covered by $C$. For $s_1$, $s_2$, and $s_3$ to be part of a shortest curve, the curve must have $s_1$ and $s_2$ as endpoints, for otherwise we would have a shortest curve that includes a segment from $s_1$ to $s_2$, causing it to be longer than $D$. In this case, the only possibility is that $s_1$ and~$s_2$ have two shortest paths, and are at distance $D$ from each other along a geodesic of length $2D$, so there are at most $2D$ points in $S$.
\end{proof}

\begin{corollary}
\label{cor:equilateral-dimension}
The equilateral dimension of a 2-dimensional complete Riemannian manifold of bounded genus $g$ is $O(g)$.
\end{corollary}

\section{Convex surfaces}
\label{sec:convex}

In this section we extend our results on Riemann manifolds to non-smooth surfaces of non-negative curvature. The restriction to non-negative curvature is motivated by \cref{ex:infinite-cone}, which shows that even a single cone point of negative curvature can be an obstacle to Erd\H{o}s--Anning theorems. For concreteness, in this section we usually define a \emph{convex surface} to mean the boundary of a convex subset of $\mathbb{R}^3$ with nonempty interior. However, we also allow a special case, which can be thought of as the limit of convex cylinders as the cylinder height goes to zero: the double covering of a convex subset $K$ of $\mathbb{R}^2$, obtained by gluing two copies of $K$ at corresponding boundary points. For instance, the Euclidean plane is a convex surface in the first of these two senses: it is the boundary of a halfspace. By results of Aleksei Pogorelov, a more abstract class of non-smooth convex manifolds, topologically equivalent to a sphere and with total Gaussian curvature~$4\pi$, can be characterized as being either boundaries of convex bounded subsets of $\mathbb{R}^3$ or double covers of convex bounded subsets of $\mathbb{R}^2$~\cite{Con-SR-06}. We include double covers in our definition of convex surfaces to be consistent with this result of Pogorelov, but we do not require the convex sets from which they are defined to be bounded. We measure the distance on a convex surface as the infimum of lengths of curves on the surface. By a compactness argument, each pair of points on the surface can be connected by a curve of length equal to this distance, so this is a geodesic metric space~\cite{Mye-TAMS-45}, and we can apply arguments involving the Lipschitz property of distances. Every convex surface, in the sense described here, is topologically a disk (like the Euclidean plane), a sphere (like a geometric sphere, polyhedron, or double-covered topological disk), or an annulus (like the boundary of an infinite cylinder).

\begin{figure}[t]
\centering\includegraphics[width=0.35\textwidth]{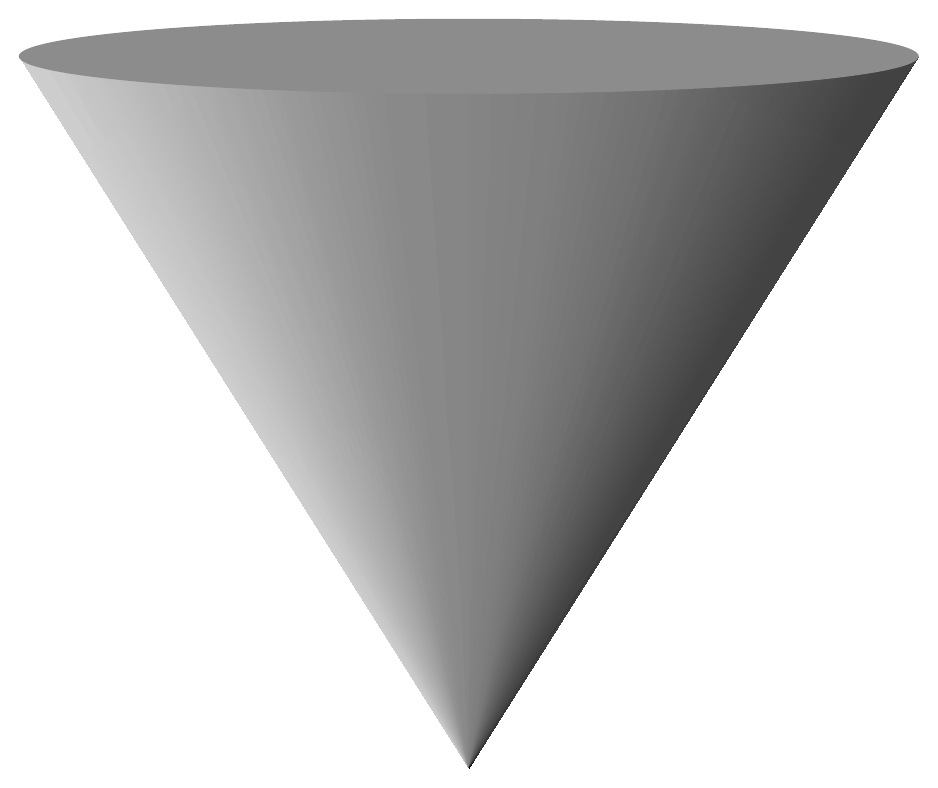}
\caption{A solid cone}
\label{fig:cone}
\end{figure}

The boundary of a (bounded) solid cone (\cref{fig:cone}) exhibits several of the common geometric features of a convex surface. The base of the cone is a flat disk; the sides of the cone have a flat surface metric (locally isometric to a subset of the Euclidean plane) even though they are curved in three-dimensional space. The circle where the base meets the sides is non-smooth; each point has a neighborhood that approaches a Euclidean metric (equivalent to that of a folded piece of paper) in the limit. The apex of the cone, however, is not Euclidean, even in the limit: it has a positive angular deficiency that can be measured by comparing the radius and circumference of small metric disks around it. The total curvature of this surface is $4\pi$, by the Gauss--Bonnet theorem, but in this case this curvature is concentrated at the cone point (where it equals the angular deficit of this point) and along the circle where the base and sides meet.

In general, when a shortest curve on a convex surface passes through a locally Euclidean point, it must pass straight through it, in the local Euclidean geometry of a neighborhood of a point. A curve that does not do this could be shortcut to bypass the point. It is not possible for a shortest curve to pass through a cone point (any point of positive angular deficiency): any curve through such a point can be shortcut. Thus, we need only consider such points as endpoints of shortest curves, not as their interior points.

\begin{lemma}
For every point $p$ on a convex surface, the disks of radius $r$ centered on that point, in the limit as $r\to 0$, approach a cone for which $p$ either has angular deficit zero ($p$ is locally Euclidean) or a well-defined positive angular deficit ($p$ is locally a cone point).
\end{lemma}

\begin{proof}
For a point $p$ on a convex surface, let $C_r(p)$ denote the cone formed by positive combinations of $p$ with the points at distance $r$ from $p$ on the surface. Then $C_r$ is monotone in~$r$: for $r'<r$, $C_r'(p)\supset C_r(p)$, by convexity. By compactness of the space of cones at a point~$p$, in the limit as $r$ goes to zero, these cones approach a limiting cone; call this limit $C_0(p)$. Let $B$ be the curve obtained by intersecting the boundary of $C_0(p)$ with a unit sphere centered at $p$. Then $B$ is convex, and lies in a half-sphere, again by convexity of the convex surface. Thus, its length is at most $2\pi$. If it is exactly $2\pi$, then $p$ is locally Euclidean; otherwise, it is locally equivalent to a cone point with deficit equal to the difference between~$2\pi$ and the length of $B$.
\end{proof}

\iffull
\subsection{Voronoi diagrams and their properties}
\label{sec:convex-vor}
For convex surfaces, we define star-shaped sets, additively weighted Voronoi diagrams, rays, Voronoi segments, and degenerate and non-degenerate sites in the same way as for Riemannian manifolds. One small difference is that a ray can terminate, by hitting a cone point. Because every point is either locally a cone point (with no geodesics through it) or locally Euclidean (with all geodesics passing straight through it) the behavior of geodesics on convex surfaces is much the same as for Riemannian manifolds: the cone points do not affect them at at all, and the non-smooth geometry of the surface at the remaining points does not affect our use of this geometry to prove properties of these spaces. We have the following results, analogous to those for Riemannian manifolds, with the same proofs:

\begin{lemma}
For additively weighted Voronoi diagrams of convex surfaces, every non-empty Voronoi cell is star-shaped with its site in its kernel.
\end{lemma}

\begin{proof}
Same as for \cref{lem:star} and \cref{lem:Riemann-star}.
\end{proof}

\begin{lemma}
Let $s_i$ be a degenerate site of an additively weighted Voronoi diagram of a convex surface, with $s_i$ belonging to the Voronoi cell $V_j$ of another site $s_j$. Let $R_j$ be a ray with apex $s_j$ containing a shortest curve from $s_i$ to $s_j$, and let $R_i$ be a ray with apex $s_i$ continuing in the same direction, away from $s_i$. Then either $s_i$ is empty, or $V_i$ is the intersection of $R_i$ with the Voronoi segment of $R_j$.
\end{lemma}

\begin{proof}
Same as for \cref{lem:Riemann-ray}.
\end{proof}

\begin{corollary}
In an additively weighted Voronoi diagram on a convex surface, every point belongs to the Voronoi cell of at least one non-degenerate site
\end{corollary}

\begin{lemma}
A site $s_i$ of an additively weighted Voronoi diagram on a convex surface is non-degenerate if and only if it is an interior point of its cell $V_i$.
\end{lemma}

\begin{proof}
Same as for \cref{lem:non-degenerate-interior} and \cref{lem:Riemann-ndgi}.
\end{proof}

\begin{lemma}
Let $s_i$ be a non-degenerate site of an additively weighted Voronoi diagram on a convex surface, let $L$ be a Voronoi segment in $V_i$ having $s_i$ as an endpoint, and let $p$ be a point in the relative interior of $L$. Then $p$ is interior to $V_i$.
\end{lemma}

\begin{proof}
Same as for \cref{lem:Riemann-interior}, which depends only on the local Euclidean geometry of the surface at an interior point of a shortest curve to argue that no curve that bends at such a point can be shortest. Here we are using the fact that no interior point of a shortest curve can be a cone point.
\end{proof}

\begin{lemma}
Let $s_i$ and $s_j$ be non-degenerate sites of an additively weighted Voronoi diagram on a convex surface. Then the interior of $V_i$ is disjoint from $V_j$, and vice versa.
\end{lemma}

\begin{proof}
Same as for \cref{lem:Riemann-non-overlapping}.
\end{proof}

\begin{lemma}
Let $s_i$ and $s_j$ be distinct non-degenerate sites of an additively weighted Voronoi diagram on a convex surface, and let $p$ and $q$ be distinct points of $V_i$ and $V_j$ respectively. Then line segments (or arcs)  $ps_i$ and $qs_j$ are disjoint.
\end{lemma}

\begin{proof}
Same as for \cref{lem:Riemann-disjoint-segments}.
\end{proof}

\begin{lemma}
Let $s_i$, $s_j$, and $s_k$ be three sites of an additively weighted Voronoi diagram on a convex surface, with none of the three sites belonging to any shortest curve between the other two sites. Then $V_i\cap V_j\cap V_k$ consists of at most two points.
\end{lemma}

\begin{proof}
Same as \cref{lem:Riemann-triples}, noting that convex surfaces have Euler genus zero.
\end{proof}
\fi

\subsection{Erd\H{o}s--Anning theorems for convex surfaces}

\begin{theorem}
\label{thm:convex-triangle}
Let $s_1$, $s_2$, and $s_3$ be three points of a convex surface, with distance function $d$. Suppose also that there is no shortest curve containing all three of these points. Then the number of points that can have integer distances to all three of $s_1$, $s_2$, and $s_3$ is $O\bigl(d(s_1,s_2)\cdot d(s_1,s_3)\bigr)$.
\end{theorem}

\begin{proof}
Same as for \cref{thm:triangle} and \cref{thm:Riemann-triangle}, substituting the properties of additive Voronoi diagrams for convex surfaces in place of the corresponding properties of the diagrams for convex distance functions and Riemannian manifolds.
\end{proof}

\begin{lemma}
Let $S$ be a set of three or more points of a convex surface such that each three points of~$S$ belong to a shortest curve.
Then $S$ is a subset of a geodesic, and every shortest curve between pairs of points in $S$ lies along this geodesic.
\end{lemma}

\begin{proof}
Same as for \cref{lem:Riemann-geodesic}, again substituting the properties of Voronoi diagrams on these surfaces. 
\end{proof}

\begin{theorem}
If a set $S$ of points of a convex surface has the property that all distances between points of $S$ are integers, then either $S$ is finite or all triples of points in $S$ are subsets of shortest curves along a single geodesic.
\end{theorem}

\begin{proof}
Same as for \cref{thm:Riemann-finite}, again substituting the properties of Voronoi diagrams on these surfaces.
\end{proof}

Our diameter bound seems unlikely to be tight. It uses the following preliminary lemmas.

\begin{figure}
\centering\includegraphics[width=0.5\textwidth]{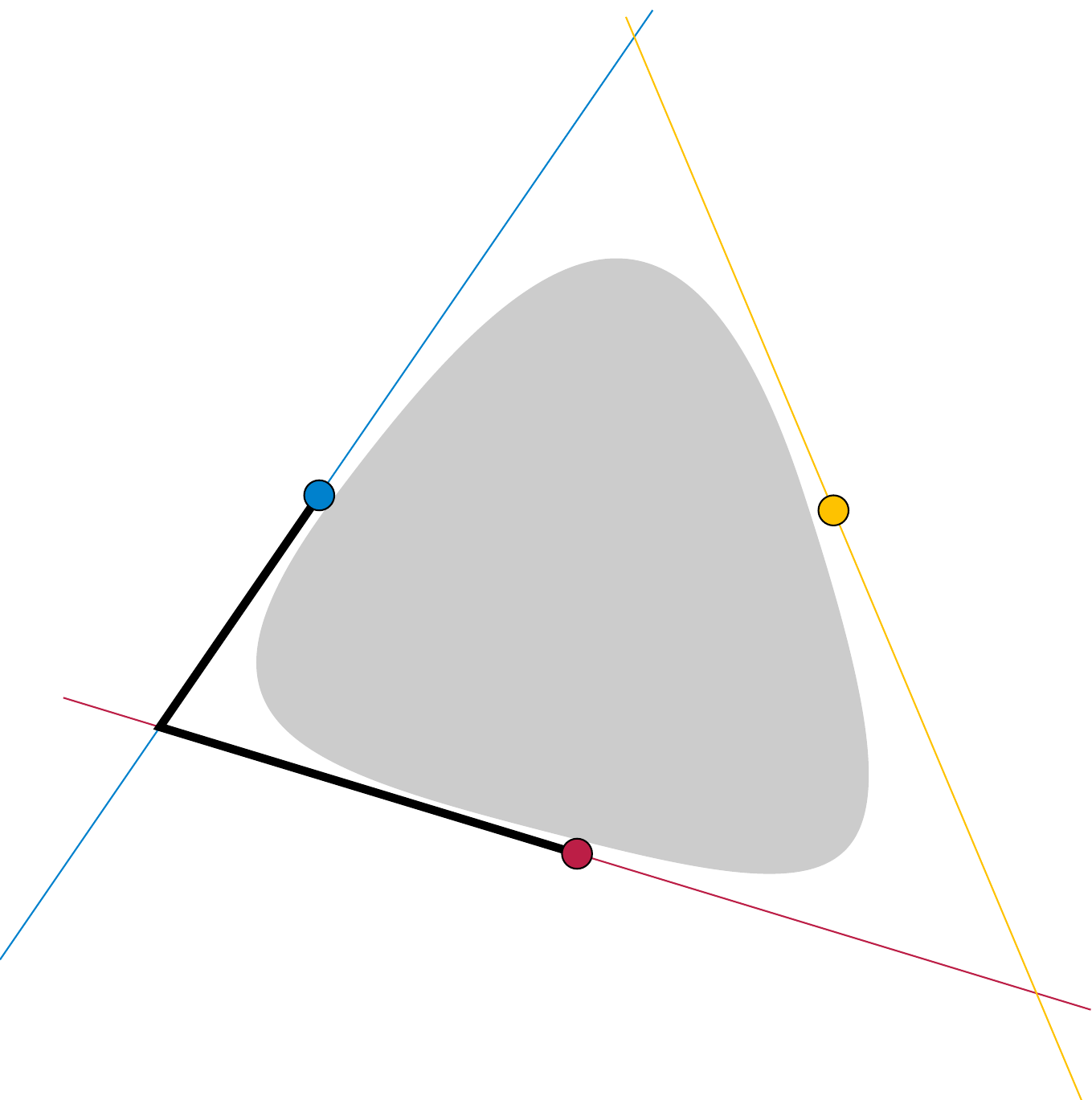}
\caption{\cref{lem:close-2-of-3}: a short path connecting two of three points separated by a convex set}
\label{fig:close-2-of-3}
\end{figure}

\begin{lemma}
\label{lem:close-2-of-3}
Suppose that three given points in $\mathbb{R}^d$ are disjoint from an open convex set $U$. Then two of the points can be connected by a curve disjoint from  $U$ of length at most twice their Euclidean distance.
\end{lemma}

\begin{proof}
Consider the plane containing the three points; if $U$ does not intersect this plane then the points can be connected directly. Within this plane draw a line through each point, disjoint from $U$. The three lines form a triangle or unbounded three-sided region containing~$U$; if all three points lie its sides, one of its internal angles is $\ge 60^\circ$, and the curve connecting two points through this angle has length at most twice the distance between the points (\cref{fig:close-2-of-3}; worst case: $U$ is an equilateral triangle and the three points are its edge midpoints). If one of the three points is not on a side of the region containing $U$, it can see the entire line through another of the points, and be connected directly to that point.
\end{proof}

\begin{lemma}
\label{lem:convex-close}
If $n$ points on a convex surface have diameter $D$ then some two of the points have geodesic distance $O(D/\sqrt{n})$.
\end{lemma}

\begin{proof}
Let the set of points be $S$, and let the convex surface be the boundary of a set $K$ (or the double cover of a set $K$), embedded in $\mathbb{R}^3$. By assumption, $S$ lies within a ball $B$ of radius~$D$ in $\mathbb{R}^3$. Subdivide the ball by a grid of cubes, of side length $O(D/\sqrt{n})$, positioned so that none of the given points lies on a cube boundary, and let $\Box$ denote the set of cubes in this grid whose intersection with $B\cap K$ is non-empty and includes at least one non-vertex point of the cube. Each point in $S$ is contained in at least one cube $C\in\Box$. This cube $C$ must touch (at least at a corner) a grid cube that is outside $\Box$, because if all eight cubes corner-adjacent to $C$ belonged to $\Box$, the intersection points of these eight cubes with $B\cap K$ would have a convex hull that entirely encloses~$C$, preventing $C$ from containing any boundary points of~$K$. Therefore, $C$ is within $O(1)$ grid positions of a boundary square of the union of cubes $\cup\Box$. But $\cup\Box$ is orthogonally convex -- any axis-parallel line can cross only two boundary squares (one in each direction) -- so it has $O(n)$ boundary squares. Therefore there are also $O(n)$ cubes near these squares that contain points of $S$. By choosing the grid size appropriately, we can adjust the constant factors in this $O(n)$ bound and ensure that fewer than $n/2$ cubes contain points of $S$. With this adjustment, some cube $C\in\Box$, of side length $O(D/\sqrt{n})$, contains at least three points of $S$. We consider two cases:
\begin{itemize}
\item If the convex surface is non-degenerate (it bounds a convex set with non-empty interior),  then by \cref{lem:close-2-of-3}, two of these three points are connected by a path in space, disjoint from the interior of the given convex set, of length $O(D/\sqrt{n})$, and hence are at geodesic distance $O(D/\sqrt{n})$ on the convex surface.
\item Otherwise, the convex surface is a double-covered two-dimensional convex set, two of the three points are on the same side of the double cover, and these two points can be connected directly within $C$ by a geodesic segment of length $O(D/\sqrt{n})$.
\end{itemize}
In either case, we have two points at geodesic distance $O(D/\sqrt{n})$.
\end{proof}

The example of a square grid of $\sqrt{n}\times\sqrt{n}$ points on a flat plane, scaled to have diameter~$D$, shows that \cref{lem:convex-close} is tight.

\begin{theorem}
\label{thm:convex-diameter}
If a set $S$ of points of a convex surface has integer distances and diameter $D$, then $|S|=O(D^{4/3})$.
\end{theorem}

\begin{proof}
Let $n$ be the size of $S$, and let its two closest such points be $p$ and $q$, with distance $O(D/\sqrt{n})$ by \cref{lem:convex-close}.  We distinguish two cases:
\begin{itemize}
\item Suppose that all triples of points in $S$ that include both $p$ and $q$ belong to shortest curves. Each point in $S$ must belong either to a ray from $p$ through $q$ or a ray from $q$ through $p$, with the ray containing a shortest curve from $p$ to $q$.
Further, this shortest curve (and hence these two rays) must be unique, for if there were two shortest curves between $p$ and $q$, one of them could be combined with part of a shortest curve from $p$ or $q$ to a third point $r$ to form a shortest curve that passes through $q$ or $p$ at a non-straight angle, an impossibility. Thus, two rays contain all remaining points in $S$, and these points lie at integer spacing along these two rays, giving $O(D)$ points in $S$ in total.
\item Otherwise, some point $r$ does not lie on a shortest curve with $p$ and $q$. By \cref{thm:convex-triangle}, there can be $O(D^2/\sqrt{n})$ points in $S$. That is, $n=O(D^2/\sqrt{n})$. Multiplying both sides by $\sqrt{n}$ and taking the $2/3$ power gives the result.\qedhere
\end{itemize}
\end{proof}

\section{Discussion}

We have shown that integer-distance point sets are finite or collinear in three broad classes of metric spaces each generalizing Euclidean distance in the plane. The next natural direction to look for a further extension of these results is in higher dimensions. Erd\H{o}s and Anning state without proof that their results extend to $d$-dimensional Euclidean space~\cite{AnnErd-BAMS-45}. However, the obvious attempt to generalize our proof to higher dimensions, by using intersection patterns of $(d+1)$-tuples of Voronoi cells, does not work for convex distance functions: there exist three-dimensional convex distance functions for which some 4-tuple of points is equidistant from arbitrarily many other points~\cite{IckKleLe-FI-95}. For Riemannian manifolds, the theorem does not generalize, as \cref{ex:high-genus} shows. Nevertheless, it would be interesting to study similar questions for higher-dimensional uniform metrics, such as hyperbolic space. It would also be of interest to strengthen the bound on diameter in \cref{thm:Riemann-diameter} and \cref{thm:convex-diameter}.

\bibliographystyle{plainurl}
\bibliography{diophantine}

\end{document}